\documentclass[journal]{IEEEtran}
\usepackage{amsmath,amsfonts,amssymb,euscript, graphicx, 
epsfig,
enumerate,float,afterpage, subfigure, ifthen}%

\usepackage{url}
\usepackage{hyperref}

\newtheorem{thm}{Theorem}

\newtheorem{lem}{Lemma}
\newtheorem{defn}{Definition}

\newcommand{\expect}[1]{\mathbb{E}\left[#1\right]}
\newcommand{\defequiv}{\mbox{\raisebox{-.3ex}{$\overset{\vartriangle}{=}$}}}

\newcommand{\norm}[1]{||{#1}||}

\newcommand{\bv}[1]{{\boldsymbol{#1} }}
\newcommand{\script}[1]{{{\cal{#1} }}}

\begin{document}

\title
  {Distributed Stochastic Optimization via Correlated Scheduling}
\author{Michael J. Neely\\University of Southern California\\\url{http://www-bcf.usc.edu/~mjneely}$\vspace{-.3in}$
\thanks{The author is with the  Electrical Engineering department at the University
of Southern California, Los Angeles, CA.} 
\thanks{This work is supported in part  by one or more of:  the NSF Career grant CCF-0747525,  NSF grant 1049541,  the 
Network Science Collaborative Technology Alliance sponsored
by the U.S. Army Research Laboratory W911NF-09-2-0053.}}

\markboth{}{Neely}

\maketitle

\begin{abstract}   
This paper considers a problem where multiple users make repeated decisions based on their own observed events.  The events and decisions at each time step determine the values of a utility function and a collection of penalty functions.  The goal is to make distributed decisions over time to maximize time average utility subject to time average constraints on the penalties.  An example is a collection of power constrained sensor nodes that repeatedly report their own observations to a fusion center. 
Maximum time average utility  is fundamentally reduced because users do not know the events observed by others.  Optimality is characterized for this distributed context.  It is shown that optimality is achieved by correlating user decisions through a commonly known pseudorandom sequence.   An optimal algorithm is developed that chooses pure strategies at each time step based on a set of time-varying weights. 
\end{abstract}

\section{Introduction}

Consider a multi-user system that operates over discrete time with unit time slots $t \in \{0, 1, 2, \ldots\}$. There are $N$ users.  At each time slot $t$, each user $i$ observes a \emph{random event} $\omega_i(t)$ and makes a \emph{control action} $\alpha_i(t)$ based on this observation.  Let $\bv{\omega}(t)$ and $\bv{\alpha}(t)$ be vectors of these values: 
\begin{eqnarray*}
\bv{\omega}(t) &=& (\omega_1(t), \omega_2(t), \ldots, \omega_N(t)) \\
\bv{\alpha}(t) &=& (\alpha_1(t), \alpha_2(t), \ldots, \alpha_N(t)) 
\end{eqnarray*}
For each slot $t$, these vectors determine the values of a \emph{system utility} $u(t)$ and a collection of \emph{system penalties} $p_1(t), \ldots, p_K(t)$ (for some non-negative integer $K$) via real-valued functions: 
\begin{eqnarray*}
u(t) &=& \hat{u}(\bv{\alpha}(t), \bv{\omega}(t)) \\
p_k(t) &=& \hat{p}_k(\bv{\alpha}(t), \bv{\omega}(t)) \: \: \forall k \in \{1, \ldots, K\} 
\end{eqnarray*} 
The functions $\hat{u}(\cdot)$ and $\hat{p}_k(\cdot)$ are arbitrary and can possibly be negative.  Negative penalties can be used to represent desirable \emph{system rewards}.  

The goal is to make distributed decisions over time that maximize time average utility subject to time average constraints on the penalties.  
Central to this problem is the assumption that each user $i$ can only observe $\omega_i(t)$, and cannot observe the value of $\omega_j(t)$ for other users $j \neq i$.   Further, each user $i$ only knows its own action $\alpha_i(t)$, but does not know the actions $\alpha_j(t)$ of others.  Therefore, each user only knows a portion of the arguments that go into the functions $\hat{u}(\bv{\alpha}(t), \bv{\omega}(t))$ and $\hat{p}_k(\bv{\alpha}(t), \bv{\omega}(t))$ for each slot $t$.  This uncertainty fundamentally restricts the time averages that can be achieved.  

Specifically, assume the random event vector $\bv{\omega}(t)$ is independent and identically distributed (i.i.d.) over slots (possibly correlated over entries in each slot).  The vector $\bv{\omega}(t)$ takes values in some abstract \emph{event space} $\Omega = \Omega_1 \times \Omega_2 \times \cdots \times \Omega_N$, where $\omega_i(t) \in \Omega_i$ for all $i \in \{1, \ldots, N\}$ and all slots $t$.  Similarly, assume $\bv{\alpha}(t)$ is chosen in some abstract \emph{action space} $\script{A} = \script{A}_1 \times \script{A}_2 \times \cdots \times \script{A}_N$, where $\alpha_i(t) \in \script{A}_i$ for all $i \in \{1, \ldots, N\}$ and all slots $t$. 
Let $\overline{u}$ and $\overline{p}_k$ be the time average expected utility and penalty incurred by a particular algorithm:\footnote{For simplicity, it is temporarily assumed that the time averages exist.  A more precise formulation is specified in Section \ref{section:optimality} using $\liminf$ and $\limsup$.} 
\begin{eqnarray*}
\overline{u} &=& \lim_{t\rightarrow\infty} \frac{1}{t}\sum_{\tau=0}^{t-1} \expect{u(\tau)} \\
\overline{p}_k &=& \lim_{t\rightarrow\infty} \frac{1}{t}\sum_{\tau=0}^{t-1}\expect{p_k(\tau)} 
\end{eqnarray*} 
The following problem is considered: 
\begin{eqnarray}
\mbox{Maximize:} & \overline{u} \label{eq:p1} \\
\mbox{Subject to:} & \overline{p}_k \leq c_k \: \: \: \forall k \in \{1, \ldots, K\} \label{eq:p2} \\
& \mbox{Decisions are distributed} \label{eq:p3} 
\end{eqnarray}
where  $c_k$ are a given collection of real numbers that specify constraints on the time average penalties.

The constraint that decisions must be distributed, specified in \eqref{eq:p3}, is not mathematically precise.  This constraint is more carefully posed in Section \ref{section:optimality}.  Without the distributed scheduling constraint, the problem \eqref{eq:p1}-\eqref{eq:p2} reduces to a standard problem of stochastic network optimization and can be solved via the \emph{drift-plus-penalty method} \cite{sno-text}. Such a centralized approach would allow users to coordinate to form an action vector $\bv{\alpha}(t)$ based on full knowledge of the event vector
$\bv{\omega}(t)$.  The time average utility achieved by the best centralized algorithm can be strictly larger than that of the best distributed algorithm. 
This is shown for an example sensor network problem in Section \ref{section:sensor-network}. 

\subsection{Applications to sensor networks} 

The above formulation is useful for a variety of stochastic network optimization problems where distributed agents make their own decisions based on partial 
system knowledge.  An important example is a network of 
wireless sensor nodes that repeatedly send reports about system events to a fusion center. 
The goal is to make distributed decisions that maximize time average \emph{quality of information}.    This scenario was previously considered by Liu et al. in  
\cite{network-corroboration-tpds}.  There, sensors can provide reports every slot $t$ using one of  multiple \emph{reporting formats}, such as text, image, or video. Sensors can also 
choose to remain idle on slot $t$. Thus, the action spaces $\script{A}_i$ are the same for all sensors $i$: 
\[ \alpha_i(t) \in \script{A}_i \defequiv \{\mbox{idle}, \mbox{text}, \mbox{image}, \text{video}\} \: \: \: \forall i \in \{1, \ldots, N\}  \]
where the notation ``$\defequiv$'' represents \emph{defined to be equal to}. 
Each format requires a different amount of power and provides a different level of quality.  For example, define $p_i(t)$ as the power incurred by sensor $i$ on slot $t$, where: 
\[ p_i(t) = \left\{ \begin{array}{ll}
0 & \mbox{ if $\alpha_i(t) = \mbox{idle}$} \\
p_{text}&\mbox{ if $\alpha_i(t) = \mbox{text}$} \\
p_{image} & \mbox{ if $\alpha_i(t) = \mbox{image}$} \\
p_{video} & \mbox{ if $\alpha_i(t) = \mbox{video}$}  
\end{array}
\right. \]
where $p_{text}$, $p_{image}$, $p_{video}$ represent powers required for each of the three reporting formats and satisfy: 
\[ 0  < p_{text} < p_{image} < p_{video} \]

Assume that $\omega_i(t)$ represents the \emph{quality} that sensor $i$ would bring to the fusion center if it reports the event it observes on slot $t$ using the video format. 
Define $f(\alpha_i(t))$ as the fraction of this quality that is achieved under format $\alpha_i(t)$: 
\[ f(\alpha_i(t))=  \left\{ \begin{array}{ll}
0 & \mbox{ if $\alpha_i(t) = \mbox{idle}$} \\
f_{text} & \mbox{ if $\alpha_i(t) = \mbox{text}$} \\
f_{image} & \mbox{ if $\alpha_i(t) = \mbox{image}$} \\
1&\mbox{ if $\alpha_i(t)=\mbox{video}$} \\
\end{array}
\right. \]
where
\[ 0 < f_{text} < f_{image} < 1 \]

The prior work \cite{network-corroboration-tpds} considers the problem of maximizing time average utility subject to a time average power constraint: 
\[ \sum_{i=1}^N \overline{p}_i \leq c \]
where $c$ is some given positive number.  Further, that work restricts to the special case when the utility function is a \emph{separable sum} of functions of user $i$ variables, such as: 
\[ u(t) = \sum_{i=1}^N f_i(\alpha_i(t))\omega_i(t) \]

Such separable utilities cannot model the realistic scenario of \emph{information saturation}, where, once a certain amount of utility is achieved on slot $t$, there is little value of having additional sensors spend power to deliver additional information on that slot.  
The current paper considers the case of \emph{arbitrary}, \emph{possibly non-separable} utility functions.  An example is: 
\[ u(t) = \min\left[\sum_{i=1}^N f(\alpha_i(t))\omega_i(t), 1\right] \] 
This means that once a total quality of $1$ is accumulated from one or more sensors on slot $t$, there is no advantage in having other sensors report information on that slot. 
This scenario is significantly more challenging to solve in a distributed context.   For example, suppose the $\omega_i(t)$ variables are binary valued, representing whether or not sensor $i$ observes an event on slot $t$.  Suppose  $\omega_1(t)=\omega_2(t)=1$.  Utility is maximized if either sensor 1 or sensor 2 decides to report in the video format.  Power is wasted if they both send video reports.   However, sensor 1 does not know the value of $\omega_2(t)$, sensor 2 does not know the value of $\omega_1(t)$, and neither sensor knows what format will be selected by the other.   

\subsection{Applications to wireless multiple access} 

The general formulation of this paper can also treat simple forms of distributed multiple access problems.   Again suppose there are $N$ wireless sensors that report to a fusion center. For each $i \in \{1, \ldots, N\}$, define $\omega_i(t)$ as the \emph{quality} that a transmission from sensor $i$ would bring to the system if it transmits on slot $t$.   Define $\alpha_i(t)$ as a binary value that is 1 if sensor $i$ transmits on slot $t$, and 0 else. Assume the network operates according to a simple collision model, where a transmission from sensor $i$ is successful on slot $t$ if and only if it is the only sensor that transmits on that slot: 
\begin{equation} \label{eq:mac} 
 u(t) = \sum_{i=1}^N\omega_i(t)\left[\alpha_i(t)\prod_{j \neq i} (1-\alpha_j(t))\right] 
 \end{equation} 
 The above utility function is non-separable.   Concurrent work in \cite{nicolo-multi-user-sensors} considers a similar utility function for wireless energy harvesting applications. 
  

\subsection{Contributions and related work} 

The framework of partial knowledge at each user is similar in spirit 
to a \emph{multi-player Bayesian game} \cite{game-theory-book}\cite{multi-agent-systems}.  
 There, the goal is to design competitive strategies that lead to a Nash equilibrium.  
 This is significantly different from the goal of the current paper.   The current paper is not concerned with competition or equilibrium.  Rather, there is a single utility function that all users desire to maximize.  Distributed  algorithms are developed to maximize time average utility subject to 
 time average penalty constraints. 

This paper shows that an optimal distributed algorithm can be designed by having users correlate their decisions through an independent source of common randomness (Section \ref{section:optimality}).   Related notions of commonly shared randomness are used in game theory to define a \emph{correlated equilibrium}, which is typically easier to compute than a standard Nash equilibrium \cite{aumann-correlated-eq1}\cite{aumann-correlated-eq2}\cite{multi-agent-systems}\cite{game-theory-book}.  For the current paper, the shared randomness is crucial for solving the distributed optimization problem.  This paper shows that optimality can be achieved by using a shared random variable with $K+1$ possible outcomes, where $K$ is the number of penalty constraints.  The solution is computable through a linear program. Unfortunately, the linear program can have a very large number of variables, even for 2-user problems.  A reduction to polynomial complexity is shown to be possible in certain cases (Section \ref{section:LP}).  This paper also develops an online algorithm that chooses pure strategies every slot based on a set of weights that are updated at the end of each slot (Section \ref{section:DPP}). The online technique is based on Lyapunov optimization concepts \cite{sno-text}\cite{now}\cite{neely-fairness-ton}. 

Much prior work on network optimization treats scenarios where it is possible to find distributed solutions with no loss of optimality.  For example, network flow problems that are described by linear or separable convex programs can be optimally solved in a distributed manner \cite{lin-shroff-cdc04}\cite{xiao-johansson-boyd-toc}\cite{low-flow-control}\cite{neely-fairness-ton}.   Problems where network nodes want to average sensor data \cite{rabbat-sensor-nets} or compute convex programs \cite{neely-dist-comp} have distributed solutions.  Work in \cite{van-roy-dist} solves for an optimal vector of parameters associated with an infinite horizon Markov decision problem using distributed agents.  
Work in \cite{jiang-entropy}\cite{shah-reversible}\cite{jiang-walrand-book} develops distributed multiple access methods that converge to optimality.  However, the above problems do not have random events that create a fundamental gap between centralized and distributed performance.  

Recent work in \cite{Nayyar-thesis} derives structural results for 
distributed optimization in Markov decision systems with delayed information.  Such problems \emph{do} exhibit gaps between centralized and distributed scheduling.  The use of \emph{private information} 
in \cite{Nayyar-thesis} is similar in spirit to the assumption in the current paper that each user observes
its own random event $\omega_i(t)$. The work \cite{Nayyar-thesis} derives a sufficient statistic for dynamic programming. 
It does not consider time average constraints and its solutions do not involve correlated scheduling via a pseudorandom sequence.   Recent work in \cite{nicolo-multi-user-sensors} considers distributed reporting of events with different qualities, but considers a more restrictive class of policies that do not use correlated scheduling. 
The current paper treats a different model than \cite{Nayyar-thesis} and \cite{nicolo-multi-user-sensors}, and 
shows that correlated scheduling is necessary in systems with constraints.
Further, the current paper provides complexity reduction results under a preferred action property (Section \ref{section:LP}) and provides an online algorithm that does not require a-priori knowledge of event probabilities (Section \ref{section:DPP}). 

\section{Example sensor network problem} \label{section:sensor-network} 

This section illustrates the benefits of using a common source of randomness
for a simple example network. 
Suppose the network has two sensors that operate over time slots $t \in \{0, 1, 2, \ldots\}$.  Every slot, the sensors observe the state of a particular system 
and choose whether or not to report their observations to a fusion center. 
Let $\omega_i(t)$ be a binary variable that is 1 if sensor $i$ observes an event on slot $t$, and $0$ else.  
Let  $\alpha_1(t)$ and $\alpha_2(t)$ be the slot $t$ decision variables, so that $\alpha_i(t)=1$ if sensor $i$ 
reports on slot $t$, and $\alpha_i(t)=0$ otherwise.  Suppose the fusion center trusts sensor 1 more than sensor 2.  
The utility $u(t)$ is: 
\[ u(t) = \min[\omega_1(t)\alpha_1(t) + \omega_2(t)\alpha_2(t)/2, 1] \]
so that the deterministic function $\hat{u}(\cdot)$ is given by: 
\begin{equation} \label{eq:example-u}
 \hat{u}(\alpha_1, \alpha_2, \omega_1, \omega_2) = \min[\omega_1\alpha_1 + \omega_2\alpha_2/2, 1] 
 \end{equation} 
Therefore, $u(t)\in\{0,1/2, 1\}$ for all slots $t$.  
If $\omega_1(t)=1$ and sensor 1 reports on slot $t$, there is no utility increase if sensor 2 also reports.

Each report uses one unit of power.   Let $p_i(t)$ be the power incurred by sensor $i$ on slot $t$, being $1$ if it reports its observation, and $0$ otherwise.  The power penalties for $i \in \{1, 2\}$ are:  
\begin{equation} \label{eq:p-example} 
 p_i(t) = \alpha_i(t)  
 \end{equation} 
so that $\hat{p}_i(\alpha_1, \alpha_2, \omega_1, \omega_2) = \alpha_i$ for $i \in \{1, 2\}$. 
Each sensor $i$ can choose \emph{not} to report an observation in order to save power.  The difficulty is that neither sensor knows what event was observed by the other.  Therefore, a distributed algorithm might send reports from \emph{both} sensors on a given slot.   A centralized scheduler would avoid this because it wastes power without increasing utility. 

Suppose that $\omega_1(t)$ and $\omega_2(t)$ are independent of each other and i.i.d. over slots, with: 
\begin{eqnarray*}
 Pr[\omega_1(t)=1] = 3/4 , & Pr[\omega_1(t)=0] = 1/4 \\
 Pr[\omega_2(t)=1] = 1/2 , & Pr[\omega_2(t)=0] = 1/2 
 \end{eqnarray*}
To fix a specific numerical example, consider the following problem: 
\begin{eqnarray}
\mbox{Maximize:} &  \overline{u} \label{eq:ex1}  \\
\mbox{Subject to:} & \overline{p}_1 \leq 1/3  \: \: , \: \: 
 \overline{p}_2 \leq1/3 \label{eq:ex2} \\
& \mbox{Decisions are distributed}\label{eq:ex3}  
\end{eqnarray}

\subsection{Independent reporting} 

Consider the following class of \emph{independent scheduling} algorithms:  Each sensor $i$ independently decides to report with probability $\theta_i$ if it observes $\omega_i(t)=1$ (it does not report if $\omega_i(t)=0$).   Since $\bv{\omega}(t)$ is i.i.d. over slots, the resulting sequences $\{u(t)\}_{t=0}^{\infty}$, $\{p_1(t)\}_{t=0}^{\infty}$, $\{p_2(t)\}_{t=0}^{\infty}$ are i.i.d. over slots.  The time averages are: 
\begin{eqnarray*}
\overline{p}_1 = \frac{3}{4}\theta_1 \: \: \: \: \:  , \: \: \: \: \: 
\overline{p}_2 = \frac{1}{2}\theta_2 
\end{eqnarray*}
\begin{eqnarray*}
\overline{u} &=& 
\expect{u(t)|\omega_1(t)=1,\omega_2(t)=0}\frac{3}{4}\frac{1}{2} \\
&& + \expect{u(t)|\omega_1(t)=0,\omega_2(t)=1}\frac{1}{4}\frac{1}{2} \\
&& + \expect{u(t)|\omega_1(t)=\omega_2(t)=1}\frac{3}{4}\frac{1}{2} \\
&=& \frac{3}{4}\frac{1}{2}\theta_1 + \frac{1}{4}\frac{1}{2}(\theta_2/2) + \frac{3}{4}\frac{1}{2}(\theta_1 + (1-\theta_1)\theta_2/2)
\end{eqnarray*}

For this class of algorithms, utility is maximized by choosing $\theta_1$ and $\theta_2$ to meet the power constraints with equality. This leads to
$\theta_1 = 4/9$, $\theta_2 = 2/3$.  The resulting utility is: 
\[ \overline{u} =  4/9  \approx 0.44444 \]

\subsection{Correlated reporting} \label{section:correlated-reports} 

As an alternative, consider the following three strategies: 
\begin{itemize} 
\item Strategy 1:  $\omega_1(t)=1 \implies \alpha_1(t)=1$ (else, $\alpha_1(t)=0$).  Sensor 2 always chooses $\alpha_2(t)=0$. 

\item Strategy 2: $\omega_2(t)=1 \implies \alpha_2(t)=1$ (else, $\alpha_2(t)=0$). Sensor 1 always chooses $\alpha_1(t)=0$. 

\item Strategy 3: $\omega_1(t)=1 \implies \alpha_1(t)=1$ (else, $\alpha_1(t)=0$).  $\omega_2(t)=1 \implies \alpha_2(t)=1$ (else, $\alpha_2(t)=0$). 
\end{itemize} 
The above three strategies are \emph{pure strategies} because $\alpha_i(t)$ is a deterministic function of $\omega_i(t)$ for each sensor $i$. Now let $X(t)$ be an external source of randomness that is commonly known at both sensors on slot $t$.   Assume $X(t)$ is independent of everything else in the system, and is i.i.d. over slots with: 
\begin{eqnarray*}
Pr[X(t)=1] &=& \theta_1 \\
Pr[X(t)=2] &=& \theta_2 \\
Pr[X(t)=3] &=& \theta_3
\end{eqnarray*}
where $\theta_1, \theta_2, \theta_3$ are probabilities that sum to 1. Consider the following algorithm:  On slot $t$, if $X(t)=m$ then choose strategy $m$, where $m \in \{1, 2, 3\}$. 
This algorithm can be implemented by letting $X(t)$ be a pseudorandom sequence that is installed in both sensors at time 0. The resulting time averages are: 
\begin{eqnarray*}
&\overline{p}_1 = (\theta_1+\theta_3)\frac{3}{4}  \: \: , \: \: \overline{p}_2 = (\theta_2+\theta_3)\frac{1}{2} \\
&\overline{u} = \theta_1\frac{3}{4} + \theta_2\frac{1}{2}\frac{1}{2} + \theta_3(\frac{3}{4} + \frac{1}{4}\frac{1}{2}\frac{1}{2})
\end{eqnarray*}
A simple linear program can be used to compute the optimal  $\theta_1, \theta_2, \theta_3$ probabilities for this algorithm structure.  
The result is $\theta_1= 1/3$, $\theta_2 = 5/9$, $\theta_3=1/9$.  
The resulting time average utility is: 
\[ \overline{u} =  23/48 \approx 0.47917 \]
This is strictly larger than the time average utility of $0.44444$ achieved by the independent reporting algorithm. Thus, performance can be strictly improved by correlating
reports via a common source of randomness.  
Alternatively, the same time averages can be achieved by \emph{time sharing}:  The two sensors agree 
to use a periodic schedule of period 9 slots.  The first $3$ slots of the period use strategy 1, the next 5 slots use strategy 2, and the final slot uses strategy 3.

\subsection{Centralized reporting}

Suppose sensors coordinate by observing $(\omega_1(t), \omega_2(t))$ and then cooperatively selecting $(\alpha_1(t), \alpha_2(t))$. It turns out that an optimal centralized
policy is as follows \cite{sno-text}:  Every slot $t$, observe $(\omega_1(t),\omega_2(t))$ and choose $(\alpha_1(t), \alpha_2(t))$ as follows: 
\begin{itemize} 
\item $(\omega_1(t), \omega_2(t))=(0,0) \implies (\alpha_1(t), \alpha_2(t))=(0,0)$.

\item $(\omega_1(t), \omega_2(t))=(0,1) \implies (\alpha_1(t), \alpha_2(t))=(0,1)$. 

\item If $(\omega_1(t), \omega_2(t))=(1,0)$, independently choose: 
\[ (\alpha_1(t), \alpha_2(t)) =  \left\{ \begin{array}{ll}
(1,0)  &\mbox{ with probability $8/9$} \\
(0,0)   & \mbox{ with probability $1/9$} 
\end{array}
\right. \]

\item If $(\omega_1(t), \omega_2(t))=(1,1)$, independently choose: 
\[ (\alpha_1(t), \alpha_2(t)) =  \left\{ \begin{array}{ll}
(0,1)   & \mbox{ with probability $5/9$} \\
(0,0) & \mbox{ with probability $4/9$}  
\end{array}
\right. \]
\end{itemize} 
The resulting optimal centralized time average utility is: 
\[ \overline{u} =  0.5  \]
This is larger than the value $0.47917$ achieved by the distributed algorithm of the previous subsection. 

The question remains:  Is it possible to construct some other distributed algorithm that yields $\overline{u} > 0.47917$? 
Results in the next section imply this is impossible.   Thus, the correlated reporting algorithm of the previous subsection 
optimizes time average utility over all possible distributed algorithms that satisfy 
the constraints.  Therefore, for this example, there is a \emph{fundamental gap} between the performance of the best centralized algorithm and the best distributed algorithm.

\section{Characterizing optimality} \label{section:optimality}

This section considers the general $N$ user problem and 
characterizes optimality over all possible distributed algorithms.   Recall that: 
\begin{eqnarray*}
\bv{\omega}(t) &\in& \Omega = \Omega_1 \times \cdots \times \Omega_N \\ 
\bv{\alpha}(t) &\in& \script{A} = \script{A}_1 \times \cdots \times \script{A}_N
\end{eqnarray*}
where the vectors $\bv{\omega}(t)$ are i.i.d. over slots (possibly correlated over entries in each slot).  
Assume that the sets $\Omega_i$ and $\script{A}_i$ are finite with sizes denoted $|\Omega_i|$ and $|\script{A}_i|$. 
For each $\bv{\omega} \in \Omega$ define: 
\[ \pi(\bv{\omega}) = Pr[\bv{\omega}(t) = \bv{\omega}] \] 

Define the \emph{history} $\script{H}(t)$ by:  
\[ \script{H}(t) \defequiv \{(\bv{\omega}(0), \bv{\alpha}(0)), \ldots, (\bv{\omega}(t-1), \bv{\alpha}(t-1))\} \]
This section considers all distributed algorithms, including those where all users know the full history $\script{H}(t)$.  Such information might be available through a feedback message that specifies $(\bv{\alpha}(t), \bv{\omega}(t))$ at the end of each slot $t$.  Theorem \ref{thm:optimality} shows that optimality can be achieved \emph{without} this history information. 

First, it is important to make the distributed scheduling constraint \eqref{eq:p3} mathematically precise.  One might attempt to use the following condition.  For all slots $t$, the decisions made by each user $i \in \{1, \ldots, N\}$ must satisfy: 
\begin{eqnarray}
Pr[\alpha_i(t) = \alpha_i | \omega_i(t)=\omega_i, \script{H}(t)] \nonumber \\
= Pr[\alpha_i(t) = \alpha_i | \bv{\omega}(t) = \bv{\omega}, \script{H}(t)] \label{eq:bad-distributed-constraint} 
\end{eqnarray}
for all vectors $\bv{\omega} = (\omega_1, \ldots, \omega_N) \in \Omega_1 \times \cdots \times \Omega_N$ and all $\alpha_i \in \script{A}_i$. 
The condition \eqref{eq:bad-distributed-constraint} specifies that $\alpha_i(t)$ is conditionally independent of $(\omega_j(t))|_{j\neq i}$ given $\omega_i(t)$, $\script{H}(t)$.  While this condition is indeed required, it turns out that it is not restrictive enough.  Appendix B provides an example utility function for which there is an algorithm that satisfies \eqref{eq:bad-distributed-constraint} but yields expected utility strictly larger than that of any ``true'' distributed algorithm (as defined in the next subsection).

\subsection{The distributed scheduling constraint} 

An algorithm for selecting $\bv{\alpha}(t)$ over slots $t \in \{0, 1, 2, \ldots\}$ is \emph{distributed} if:
\begin{itemize} 
\item There is an abstract set $\script{X}$, called a \emph{common information set}. 
\item There is a sequence of \emph{commonly known 
random elements} $X(t) \in \script{X}$ such that $\bv{\omega}(t)$ is independent of $X(t)$ 
for each $t \in \{0, 1, 2, \ldots\}$.
\item There are deterministic functions $f_{i}(\omega_i, X)$ for each $i \in \{1, \ldots, N\}$ of the form: 
\[ f_{i}: \Omega_i \times \script{X} \rightarrow \script{A}_i \]
\item The decisions $\alpha_i(t)$ satisfy the following for all slots $t$: 
\begin{equation} \label{eq:distributed-constraint}
\alpha_i(t) = f_{i}(\omega_i(t), X(t)) \: \:  \mbox{for all $i \in \{1, \ldots, N\}$}
\end{equation} 
\end{itemize} 

The above definition includes a wide class of algorithms.  Intuitively, the random elements $X(t)$ can be designed as any source of 
common randomness on which users can base their decisions.  For example, $X(t)$ can be designed to have the form: 
\[ X(t) = (t, \script{H}(t), Y(t)) \]
where $Y(t)$ is a random element with support and distribution that can possibly depend on $\script{H}(t)$ as well as past values $Y(\tau)$ for $\tau < t$.   The only restriction is that $X(t)$ is independent of $\bv{\omega}(t)$.  Because the $\bv{\omega}(t)$ vectors are i.i.d. over slots, $X(t)$ can be based on any events that occur before slot $t$.  

\subsection{The optimization problem} 

For notational convenience, define: 
\begin{eqnarray*}
p_0(t) &\defequiv& -u(t)\\
 \hat{p}_0(\bv{\alpha}(t), \bv{\omega}(t)) &\defequiv& -\hat{u}(\bv{\alpha}(t), \bv{\omega}(t)) 
 \end{eqnarray*}
Maximizing the time average expectation of $u(t)$ is equivalent to minimizing the time average expectation of $p_0(t)$.  
For each $k \in \{0, 1, \ldots, K\}$ and each slot $t>0$ define: 
\begin{eqnarray*}
\overline{p}_k(t) \defequiv \frac{1}{t}\sum_{\tau=0}^{t-1} \expect{p_k(\tau)} 
\end{eqnarray*}
The goal is to design a distributed algorithm that solves the following: 
\begin{eqnarray}
\mbox{Minimize:} & \limsup_{t\rightarrow\infty} \overline{p}_0(t) \label{eq:q1} \\
\mbox{Subject to:} & \limsup_{t\rightarrow\infty}\overline{p}_k(t) \leq c_k \: \: \forall k \in \{1, \ldots, K\} \label{eq:q2} \\
& \mbox{Condition \eqref{eq:distributed-constraint} holds $\forall t \in \{0, 1, 2, \ldots\}$} \label{eq:q3} 
\end{eqnarray}

It is assumed throughout this paper that the constraints \eqref{eq:q2}-\eqref{eq:q3} are \emph{feasible}.  Define $p_0^{opt}$ as the infimum of all limiting $\overline{p}_0(t)$ values \eqref{eq:q1} achievable by algorithms that satisfy the constraints \eqref{eq:q2}-\eqref{eq:q3}.  The infimum is finite because $p_0(t)$ takes values in the same bounded set for all slots $t$.

\subsection{Optimality via correlated scheduling}

 A \emph{pure strategy} is defined as a vector-valued function: 
 \[ \bv{g}(\bv{\omega}) = (g_1(\omega_1), g_2(\omega_2), \ldots, g_N(\omega_N)) \] 
where $g_i(\omega_i) \in \script{A}_i$ for all $i \in \{1, \ldots, N\}$ and all  $\omega_i \in \Omega_i$. The function $\bv{g}(\bv{\omega})$ specifies a distributed decision rule where each user $i$ chooses $\alpha_i$ as a deterministic function of $\omega_i$.  Specifically, $\alpha_i = g_i(\omega_i)$. The total number of pure strategy functions $\bv{g}(\bv{\omega})$ is $\prod_{i=1}^N|\script{A}_i|^{|\Omega_i|}$.  Define $M$ as this number, and enumerate all these vectors by $\bv{g}^{(m)}(\bv{\omega})$ for $m \in \{1, \ldots, M\}$.  For each $m \in \{1, \ldots, M\}$ and $k \in \{0, 1, \ldots, K\}$ define: 
\begin{equation} \label{eq:rm} 
 r_k^{(m)} \defequiv \sum_{\bv{\omega}\in\Omega} \pi(\bv{\omega})\hat{p}_k(\bv{g}^{(m)}(\bv{\omega}), \bv{\omega}) 
 \end{equation} 
The value  $r_k^{(m)}$ is the expected value of $p_k(t)$  given that users implement strategy $\bv{g}^{(m)}(\bv{\omega})$ on slot $t$. 

Consider a randomized algorithm that, every slot $t$, independently uses strategy $\bv{g}^{(m)}(\bv{\omega})$ with 
probability $\theta_m$.   For each $k \in \{0, 1, \ldots, K\}$, the expected penalty $\expect{p_k(t)}$ under such a strategy is: 
\begin{eqnarray*}
 \expect{p_k(t)} &=& \sum_{m=1}^M\theta_m\expect{\hat{p}_k\left(\bv{g}^{(m)}(\bv{\omega}(t)), \bv{\omega}(t)\right)} \\
 &=& \sum_{m=1}^M\theta_m r_k^{(m)} 
\end{eqnarray*}
The following linear program optimizes over the $\theta_m$ probabilities for this specific algorithm structure:

\begin{eqnarray}
\mbox{Minimize:} && \sum_{m=1}^M\theta_m r_0^{(m)} \label{eq:lp1}  \\
\mbox{Subject to:} && \sum_{m=1}^M\theta_m r_k^{(m)} \leq c_k \: \: \: \: \forall k \in \{1, \ldots, K\} \label{eq:lp2}  \\
&& \theta_m \geq 0 \: \: \: \: \forall m \in \{1, \ldots, M\} \label{eq:lp3} \\
&& \sum_{m=1}^M \theta_m = 1  \label{eq:lp4} 
\end{eqnarray}
The objective \eqref{eq:lp1} corresponds to minimizing $\expect{p_0(t)}$, the constraints \eqref{eq:lp2} ensure 
$\expect{p_k(t)} \leq c_k$ for $k \in \{1, \ldots, K\}$, and the constraints \eqref{eq:lp3}-\eqref{eq:lp4} ensure the $\theta_m$ values form a valid probability mass function. 
Such a randomized algorithm does not use the history $\script{H}(t)$. 
The next theorem shows this algorithm structure is optimal.

\begin{thm} \label{thm:optimality} Suppose the problem \eqref{eq:q1}-\eqref{eq:q3} is feasible.  Then 
the linear program \eqref{eq:lp1}-\eqref{eq:lp4} is feasible, and the optimal objective value \eqref{eq:lp1} is 
equal to $p_0^{opt}$.  Furthermore, 
there exist probabilities $(\theta_1, \ldots, \theta_{M})$ that solve the linear program and satisfy 
$\theta_m>0$ for at most $K+1$ values of $m \in \{1, \ldots, M\}$. 
\end{thm} 

\begin{proof} 
See Appendix A. 
\end{proof}

\section{Reduced complexity} \label{section:LP} 

The linear 
program \eqref{eq:lp1}-\eqref{eq:lp4} uses variables $(\theta_1, \theta_2, \ldots, \theta_M)$, where $M$ is the 
number of pure strategies: 
\[ M = \prod_{i=1}^N|\script{A}_i|^{|\Omega_i|} \]

The 2-user sensor network example from Section \ref{section:sensor-network} has $|\script{A}_i|=|\Omega_i| = 2$ for $i \in \{1, 2\}$, for a total of $2^2=4$ strategy functions $g_i(\omega_i)$ for each user---hence a total of $M=16$ functions $\bv{g}(\bv{\omega})$.  However, for each user $i$, the two strategy functions $g_i(\omega_i)$ that give $g_i(0)=1$ can be removed from consideration (as it is useless for user $i$ to report if it observes no event). Thus, the effective number of strategy functions $g_i(\omega_i)$ for each user is only two, leaving only four  functions $\bv{g}(\bv{\omega})=(g_1(\omega_1), g_2(\omega_2))$.  The optimal probabilities for switching between these four is given in Section \ref{section:correlated-reports}, where it is seen that only $K+1=3$ strategies have non-zero probabilities. 

For general problems,  the value of $M$ can be very large.  The remainder of this section shows that, if certain conditions hold, the set of strategy functions can be pruned to a smaller set without loss of optimality. 
 For example, consider a two-user problem with binary actions, so that $|\script{A}_i|= 2$ for $i \in \{1, 2\}$.  Then: 
\[ M = 2^{|\Omega_1|+ |\Omega_2|} \]
If certain conditions hold, strategies can be restricted to a set of size $\tilde{M}$, where: 
\[ \tilde{M} = (|\Omega_1|+1)(|\Omega_2|+1) \]
Thus, an exponentially large set is pruned to a smaller set with polynomial size.

\subsection{The preferred action property}

Suppose the sets $\script{A}_i$ and $\Omega_i$ for each user $i \in \{1, \ldots, N\}$ are given by: 
\begin{eqnarray}
\script{A}_i &=& \{0, 1, \ldots, |\script{A}_i|-1\} \label{eq:ai-set}  \\
\Omega_i &=& \{0, 1, \ldots, |\Omega_i|-1\} \label{eq:omegai-set} 
\end{eqnarray}
For notational convenience, for each $i \in \{1, \ldots, N\}$ let $[\bv{\alpha}_{\overline{i}},\alpha_i]$ denote the $N$-dimensional
vector $\bv{\alpha} = (\alpha_1, \ldots, \alpha_N)$, where $\bv{\alpha}_{\overline{i}}$ is the $(N-1)$-dimensional vector of  $\alpha_j$ components for  $j \neq i$. This notation facilitates  comparison of two vectors that differ in just one coordinate. 
Define $\script{A}_{\overline{i}}$ and $\Omega_{\overline{i}}$ as the set of all possible $(N-1)$-dimensional vectors $\bv{\alpha}_{\overline{i}}$ and $\bv{\omega}_{\overline{i}}$, respectively.

\begin{defn} A penalty function $\hat{p}(\bv{\alpha}, \bv{\omega})$ has  the \emph{preferred action property} if for all $i \in \{1, \ldots, N\}$, 
all $\bv{\alpha}_{\overline{i}} \in \script{A}_{\overline{i}}$, and all $\bv{\omega}_{\overline{i}} \in \Omega_{\overline{i}}$, one has: 
\begin{eqnarray*}
\hat{p}([\bv{\alpha}_{\overline{i}}, \alpha], [\bv{\omega}_{\overline{i}}, \omega]) - \hat{p}([\bv{\alpha}_{\overline{i}}, \beta], [\bv{\omega}_{\overline{i}}, \omega]) \\
\geq 
\hat{p}([\bv{\alpha}_{\overline{i}}, \alpha], [\bv{\omega}_{\overline{i}}, \gamma]) - \hat{p}([\bv{\alpha}_{\overline{i}}, \beta], [\bv{\omega}_{\overline{i}}, \gamma]) 
\end{eqnarray*}
whenever $\alpha, \beta$ are values in $\script{A}_i$ that satisfy $\alpha > \beta$, and $\omega, \gamma$ are values in $\Omega_i$ that 
satisfy $\omega < \gamma$. 
\end{defn}

Intuitively, the above definition means that if user $i$ compares the difference in penalty under the actions $\alpha_i(t)=\alpha$ and $\alpha_i(t)=\beta$ (where $\alpha > \beta$), 
this difference is non-increasing in the user $i$ observation $\omega_i(t)$ (assuming all other actions and events $\bv{\alpha}_{\overline{i}}$ and $\bv{\omega}_{\overline{i}}$ are held fixed).

For example, any function $\hat{p}(\bv{\alpha}, \bv{\omega})$ that does not depend on $\bv{\omega}$ trivially satisfies the preferred action property.   This is the case for the $\hat{p}_1(\cdot)$ and $\hat{p}_2(\cdot)$ functions in \eqref{eq:p-example} used to represent power expenditures for the sensor network example of Section \ref{section:sensor-network}.  Further, the utility function \eqref{eq:example-u} in that example yields $\hat{p}_0(\cdot) = -\hat{u}(\cdot)$ that satisfies the preferred action property, as shown by the next lemma.  

\begin{lem} \label{lem:preferred1} Suppose $\script{A}_i = \{0, 1\}$ for $i \in \{1, \ldots, N\}$, $\Omega_i$ is given by 
\eqref{eq:omegai-set},  and define: 
\[ \hat{u}(\bv{\alpha}, \bv{\omega}) = \min\left[\sum_{i=1}^N\phi_i(\omega_i)\alpha_i , b\right] \]
for some (real-valued) constant $b$ and some (real-valued) non-decreasing functions $\phi_i(\omega_i)$.  Then 
the penalty function $\hat{p}_0(\bv{\alpha}, \bv{\omega}) = -\hat{u}(\bv{\alpha}, \bv{\omega})$ has the preferred action 
property. 
\end{lem}

\begin{lem} \label{lem:preferred2} 
Suppose $\script{A}_i = \{0, 1\}$ for $i \in \{1, \ldots, N\}$, $\Omega_i$ is given by \eqref{eq:omegai-set}, and define the utility function $\hat{u}(\bv{\alpha}, \bv{\omega})$ according to the multi-access example equation \eqref{eq:mac}. Then the penalty function $\hat{p}_0(\bv{\alpha}, \bv{\omega}) = -\hat{u}(\bv{\alpha}, \bv{\omega})$ has the preferred action property. 
\end{lem}

\begin{lem} \label{lem:preferred3} Suppose $\script{A}_i$ and $\Omega_i$ are given by \eqref{eq:ai-set}-\eqref{eq:omegai-set}. Define  $\hat{p}(\bv{\alpha} , \bv{\omega})$ by: 
\[ \hat{p}(\bv{\alpha}, \bv{\omega}) = \prod_{i=1}^N\phi_i(\omega_i)\psi_i(\alpha_i)   \]
where $\phi_i(\omega_i)$, $\psi_i(\alpha_i)$ are non-negative functions for all $i \in \{1, \ldots, N\}$. 
Suppose that for each $i \in \{1, \ldots, N\}$, $\phi_i(\omega_i)$ is non-increasing in $\omega_i$ and $\psi_i(\alpha_i)$ is non-decreasing in $\alpha_i$. Then $\hat{p}(\bv{\alpha}, \bv{\omega})$ has the preferred action property. 
\end{lem}

\begin{lem}\label{lem:preferred4}  Suppose $\script{A}_i$ and $\Omega_i$ are given by \eqref{eq:ai-set}-\eqref{eq:omegai-set}. Suppose $\hat{p}_1(\bv{\alpha}, \bv{\omega}), \ldots, \hat{p}_R(\bv{\alpha}, \bv{\omega})$ are a collection of functions that have the preferred action property (where $R$ is a given positive integer). Then for any non-negative weights $w_1, \ldots, w_R$, the following function has the preferred action property: 
\[ \hat{p}(\bv{\alpha}, \bv{\omega}) = \sum_{r=1}^R w_r\hat{p}_r(\bv{\alpha}, \bv{\omega}) \]
\end{lem} 

The proofs of Lemmas \ref{lem:preferred1}-\ref{lem:preferred4} are given in Appendix C. 
 
\subsection{Independent events and reduced complexity}

Consider the special case when the components of $\bv{\omega}(t) = (\omega_1(t), \ldots, \omega_N(t))$ are mutually independent, so that: 
\begin{eqnarray}
\pi(\bv{\omega}) = \prod_{i=1}^Nq_i(\omega_i) \label{eq:independence} 
\end{eqnarray}
where: 
\[ q_i(\omega_i) \defequiv Pr[\omega_i(t)=\omega_i] \]
Without loss of generality, assume $q_i(\omega_i)>0$ for all $i \in \{1, \ldots, N\}$ and all $\omega_i \in \Omega_i$. 
Recall that a pure strategy $\bv{g}(\bv{\omega})$ is composed of individual \emph{strategy functions} $g_i(\omega_i)$ for each user $i$: 
\[ \bv{g}(\bv{\omega}) = (g_1(\omega_1), \ldots, g_N(\omega_N)) \]

\begin{thm} \label{thm:non-decreasing} (Non-decreasing strategy functions) Suppose $\script{A}_i$ and $\Omega_i$ are given by \eqref{eq:ai-set}-\eqref{eq:omegai-set}. If all penalty functions $\hat{p}_k(\bv{\alpha}, \bv{\omega})$ for $k \in \{0, 1, \ldots, K\}$ have the preferred action property, and if the random event process $\bv{\omega}(t)$ satisfies
the independence property \eqref{eq:independence}, 
then it suffices to restrict attention to strategy functions $g_i(\omega_i)$ that are non-decreasing in $\omega_i$. 
\end{thm} 

\begin{proof} 
The proof uses an interchange argument.  Fix $m \in \{1, \ldots, M\}$,  $i \in \{1, \ldots, N\}$, and fix two elements $\omega$ and $\gamma$ in $\Omega_i$  
that satisfy $\omega < \gamma$.   Suppose the linear program \eqref{eq:lp1}-\eqref{eq:lp4} places weight $\theta_m>0$
on a strategy function $\bv{g}^{(m)}(\bv{\omega})$ that satisfies $g_i^{(m)}(\omega) > g_i^{(m)}(\gamma)$ (so the non-decreasing requirement is violated). 
The goal is to show this can be replaced by new strategies that do not violate the non-decreasing requirement for elements $\omega$ and $\gamma$, without loss of optimality.  

Define $\alpha = g_i^{(m)}(\omega)$ and $\beta = g_i^{(m)}(\gamma)$. Then $\alpha > \beta$. 
Define two new functions: 
\begin{eqnarray*}
g_i^{(m), low}(\omega_i) &=& \left\{\begin{array}{ll}
g_i^{(m)}(\omega_i)  &\mbox{ if $\omega_i \notin \{\omega, \gamma\}$} \\
\beta  & \mbox{ if $\omega_i \in \{\omega, \gamma\}$} 
\end{array}
\right. \\
g_i^{(m), high}(\omega_i) &=& \left\{ \begin{array}{ll}
 g_i^{(m)}(\omega_i) &\mbox{ if $\omega_i \notin \{\omega, \gamma\}$} \\
\alpha  & \mbox{ if $\omega_i \in \{\omega, \gamma\}$}
\end{array}
\right. 
\end{eqnarray*}
Unlike the original function $g_i^{(m)}(\omega_i)$, these new functions satisfy: 
\begin{eqnarray*}
g_i^{(m), low}(\omega) &\leq& g_i^{(m), low}(\gamma)  \\
g_i^{(m), high}(\omega) &\leq& g_i^{(m), high}(\gamma)
\end{eqnarray*}
Define $\bv{g}^{(m), low}(\bv{\omega})$ and $\bv{g}^{(m), high}(\bv{\omega})$ by replacing the $i$th component function $g_i^{(m)}(\omega_i)$ of
$\bv{g}^{(m)}(\bv{\omega})$ with new component functions $g_i^{(m), low}(\omega_i)$ and $g_i^{(m), high}(\omega_i)$, respectively.  
Let $p_k^{old}(t)$ be the $k$th penalty incurred in the (old) strategy that uses $\bv{g}^{(m)}(\bv{\omega})$ with probability $\theta_m$. Let $p_k^{new}(t)$ be the corresponding penalty under a (new) 
strategy that, instead of using $\bv{g}^{(m)}(\bv{\omega})$ with probability $\theta_m$, uses: 
\begin{itemize} 
\item $\bv{g}^{(m), low}(\bv{\omega})$ with probability $\theta_mq_i(\gamma)/(q_i(\omega) + q_i(\gamma))$. 
\item $\bv{g}^{(m), high}(\bv{\omega})$ with probability $\theta_mq_i(\omega)/(q_i(\omega) + q_i(\gamma))$. 
\end{itemize} 

Let $\bv{\omega}_{\overline{i}}(t)$ denote the $(N-1)$-dimensional vector of components $\omega_j(t)$ for $j \neq i$. 
Fix any vector $\bv{\omega}_{\overline{i}} \in \Omega_{\overline{i}}$.  Define $\bv{\alpha}_{\overline{i}}$ as the corresponding $(N-1)$-dimensional vector of $g_j^{(m)}(\omega_j)$ values for $j \neq i$.   Then: 
\begin{itemize} 
\item If $\bv{\omega}_{\overline{i}}(t)=\bv{\omega}_{\overline{i}}$, $\omega_i(t)=\omega$, and $\bv{g}^{(m), low}(\bv{\omega})$ is used by the new strategy, then
$\bv{\omega}(t) = [\bv{\omega}_{\overline{i}}, \omega]$ and: 
\begin{eqnarray*}
p_k^{new}(t) &=& \hat{p}_k\left(\bv{g}^{(m), low}\left([\bv{\omega}_{\overline{i}}, \omega]\right), [\bv{\omega}_{\overline{i}}, \omega]\right) \\
&=& \hat{p}_k\left([\bv{\alpha}_{\overline{i}}, \beta] , [\bv{\omega}_{\overline{i}}, \omega]  \right) 
\end{eqnarray*}
Further, since the old strategy used $g_i^{(m)}(\omega) = \alpha$: 
\begin{eqnarray*}
p_k^{old}(t) &=& \hat{p}_k\left(\bv{g}^{(m)}\left([\bv{\omega}_{\overline{i}}, \omega]\right), [\bv{\omega}_{\overline{i}}, \omega]\right) \\
&=& \hat{p}_k\left([\bv{\alpha}_{\overline{i}}, \alpha] , [\bv{\omega}_{\overline{i}}, \omega]  \right) 
\end{eqnarray*}

\item If $\bv{\omega}_{\overline{i}}(t)=\bv{\omega}_{\overline{i}}$, $\omega_i(t)=\gamma$, and $\bv{g}^{(m), high}(\bv{\omega})$ is used by the new strategy, then
$\bv{\omega}(t) = [\bv{\omega}_{\overline{i}}, \gamma]$ and: 
\begin{eqnarray*}
p_k^{new}(t) &=& \hat{p}_k\left(\bv{g}^{(m), high}\left([\bv{\omega}_{\overline{i}}, \gamma]\right), [\bv{\omega}_{\overline{i}}, \gamma]\right) \\
&=& \hat{p}_k\left([\bv{\alpha}_{\overline{i}}, \alpha] , [\bv{\omega}_{\overline{i}}, \gamma]  \right) 
\end{eqnarray*}
Further, since the old strategy used $g_i^{(m)}(\gamma) = \beta$: 
\begin{eqnarray*}
p_k^{old}(t) &=& \hat{p}_k\left(\bv{g}^{(m)}\left([\bv{\omega}_{\overline{i}}, \omega]\right), [\bv{\omega}_{\overline{i}}, \omega]\right) \\
&=& \hat{p}_k\left([\bv{\alpha}_{\overline{i}}, \beta] , [\bv{\omega}_{\overline{i}}, \gamma]  \right) 
\end{eqnarray*}

\item Suppose $\bv{\omega}_{\overline{i}}(t)=\bv{\omega}_{\overline{i}}$, but 
neither of the above two events are satisfied on slot $t$.  That is, neither of the events $\script{E}_1$ or $\script{E}_2$ are true, where: 
\begin{eqnarray*}
\script{E}_1 &\defequiv& \{\omega_i(t)=\omega \} \cap \{\bv{g}^{(m), low}(\bv{\omega}) \mbox{ is used} \} \\
\script{E}_2 &\defequiv& \{\omega_i(t)=\gamma \} \cap \{\bv{g}^{(m), high}(\bv{\omega}) \mbox{ is used} \} 
\end{eqnarray*}
Then $p_k^{new}(t) - p_k^{old}(t)=0$.
\end{itemize} 
It follows that: 
\begin{eqnarray}
&&\hspace{-.4in}\expect{p_k^{new}(t) - p_k^{old}(t)|\bv{\omega}_{\overline{i}}(t) = \bv{\omega}_{\overline{i}}} \nonumber \\
&=&\theta_mq_i(\omega)\left(\frac{q_i(\gamma)}{q_i(\omega)+q_i(\gamma)}\right) \times \nonumber \\
&& \left[ \hat{p}_k\left([\bv{\alpha}_{\overline{i}}, \beta], [\bv{\omega}_{\overline{i}}, \omega]  \right)  - \hat{p}_k\left([\bv{\alpha}_{\overline{i}}, \alpha], [\bv{\omega}_{\overline{i}}, \omega]   \right)   \right] \nonumber \\
&&+\theta_mq_i(\gamma)\left(\frac{q_i(\omega)}{q_i(\omega)+q_i(\gamma)}\right) \times \nonumber \\
&&  \left[ \hat{p}_k\left([\bv{\alpha}_{\overline{i}}, \alpha], [\bv{\omega}_{\overline{i}}, \gamma]  \right)  - \hat{p}_k\left([\bv{\alpha}_{\overline{i}}, \beta], [\bv{\omega}_{\overline{i}}, \gamma]   \right)   \right]\label{eq:big}
\end{eqnarray}
where the above uses the fact that $\omega_i(t)$ is independent of $\bv{\omega}_{\overline{i}}(t)$, so conditioning on $\bv{\omega}_{\overline{i}}(t)=\bv{\omega}_{\overline{i}}$ does not change the distribution of $\omega_i(t)$.   Because $\hat{p}_k(\cdot)$ satisfies the preferred action property and $\alpha > \beta$, $\omega < \gamma$, one has: 
\begin{eqnarray*}
\left[ \hat{p}_k\left([\bv{\alpha}_{\overline{i}}, \alpha], [\bv{\omega}_{\overline{i}}, \omega]  \right)  - \hat{p}_k\left([\bv{\alpha}_{\overline{i}}, \beta], [\bv{\omega}_{\overline{i}}, \omega]   \right)   \right] \\
\geq 
\left[ \hat{p}_k\left([\bv{\alpha}_{\overline{i}}, \alpha], [\bv{\omega}_{\overline{i}}, \gamma]  \right)  - \hat{p}_k\left([\bv{\alpha}_{\overline{i}}, \beta], [\bv{\omega}_{\overline{i}}, \gamma]   \right)   \right]
\end{eqnarray*}
and hence 
 \eqref{eq:big} is less than or equal to zero. This holds when conditioning on all possible
values of $\bv{\omega}_{\overline{i}}(t)$, and so: 
\[ \expect{p_k^{new}(t) - p_k^{old}(t)} \leq 0 \]
This holds for all penalties $k \in \{0, 1, \ldots, K\}$, and so the modified algorithm still satisfies all constraints with an optimal value for $\expect{p_0(t)}$. The interchange can be repeated a finite number of times 
until all strategy functions are non-decreasing. 
\end{proof} 

In the special case of binary actions, so that $\script{A}_i = \{0,1\}$ for all $i\in \{1, \ldots, N\}$, all non-decreasing strategy functions $g_i(\omega_i)$ have the following form: 
\begin{equation} \label{eq:threshold-function} 
g_i(\omega_i)  = \left\{ \begin{array}{ll}
0 & \mbox{ if $\omega_i < h_i^*$} \\
1  & \mbox{ if $\omega_i \geq h_i^*$} 
\end{array}
\right.
\end{equation} 
for some threshold $h_i^* \in \{0, 1, \ldots, |\Omega_i|\}$. 
There are $|\Omega_i|+1$ such threshold functions, whereas the total number of strategy functions for user $i$ is 
$2^{|\Omega_i|}$.  Restricting to the threshold functions significantly decreases complexity. 

\section{Online optimization} \label{section:DPP}

This section presents a dynamic algorithm to solve the problem \eqref{eq:q1}-\eqref{eq:q3}. The algorithm can also be viewed as an 
online solution to the linear program \eqref{eq:lp1}-\eqref{eq:lp4}.  Let $\tilde{M}$ be the number of pure strategies required 
for consideration in the linear program (where $\tilde{M}$ is possibly smaller than $M$, as discussed in the previous section). 
Reorder the functions $g^{(m)}(\bv{\omega})$  if necessary so that  every slot $t$, the system chooses a strategy function
in the set $\{g^{(1)}(\bv{\omega}), \ldots, g^{(\tilde{M})}(\bv{\omega})\}$.

Suppose all users receive feedback specifying the values of the penalties $p_1(t), \ldots, p_K(t)$ at the end of slot $t+D$, where $D$ is a non-negative
integer that represents a system delay.  For
each constraint $k \in \{1, \ldots, K\}$, define a \emph{virtual queue} $Q_k(t)$ and initialize $Q_k(0)$ to a commonly known value (typically 0). For each $t \in \{0, 1, 2, \ldots\}$ the queue is updated
by: 
\begin{equation} \label{eq:q-update} 
Q_k(t+1) = \max[Q_k(t) + p_k(t-D) - c_k, 0] 
\end{equation}
Each user can iterate the above equation based on information available at the end of slot $t$.  Thus, all users know the value of $Q_k(t)$ at the beginning of each slot $t$.
If $D>0$, define $p_k(-1)=p_k(-2) = \cdots = p_k(-D)=0$.

\begin{lem} \label{lem:vq} Under any decision rule for choosing strategy functions over time,  for all $t>0$ one has: 
\[ \frac{1}{t}\sum_{\tau=0}^{t-1} \expect{p_k(\tau-D)} \leq c_k + \frac{\expect{Q_k(t)}}{t} - \frac{\expect{Q_k(0)}}{t} \]
\end{lem} 

\begin{proof} 
From \eqref{eq:q-update} the following holds for all slots $\tau \in \{0, 1, 2, \ldots\}$: 
\[ Q_k(\tau+1) \geq Q_k(\tau) + p_k(\tau-D) - c_k \]
Thus: 
\[ Q_k(\tau+1) - Q_k(\tau) \geq p_k(\tau-D) - c_k \]
Summing over $\tau \in \{0, 1, \ldots, t-1\}$ for $t>0$ gives: 
\[ Q_k(t) - Q_k(0) \geq \sum_{\tau=0}^{t-1} p_k(\tau-D) - c_kt \]
Rearranging terms proves the result. 
\end{proof}

Lemma \ref{lem:vq} ensures the constraints \eqref{eq:q2} are satisfied 
whenever the condition $\lim_{t\rightarrow 0} \expect{Q_k(t)}/t = 0$ holds for all $k \in \{1, \ldots, K\}$, a condition called
\emph{mean rate stability} \cite{sno-text}.  

\subsection{Lyapunov optimization} 

Define $\bv{Q}(t) = (Q_1(t), \ldots, Q_K(t))$. Define $L(t)$ as the squared norm of $\bv{Q}(t)$ (divided by 2 for convenience later): 
\[ L(t) \defequiv \frac{1}{2}\norm{\bv{Q}(t)}^2 = \frac{1}{2}\sum_{k=1}^KQ_k(t)^2 \]
Define  $\Delta(t) \defequiv L(t+1) - L(t)$, called the \emph{Lyapunov drift}. 
Consider the following structure for the control decisions:  Every slot $t$ the queues $\bv{Q}(t)$ are 
observed.  Then a collection of non-negative values $\beta_m(t)$ are created that satisfy $\sum_{m=1}^{\tilde{M}} \beta_m(t)=1$ (if desired, the $\beta_m(t)$ values can 
be chosen as a function of the $\bv{Q}(t)$ values). Then an index $m \in \{1, \ldots, \tilde{M}\}$ is randomly and independently chosen according to the probability mass function $\beta_m(t)$, and the decision rule $\bv{g}^{(m)}(\bv{\omega}(t))$ is used for slot $t$.  Thus, a specific algorithm with this structure is determined by specifying how the 
 $\beta_m(t)$ probabilities are chosen on each slot $t$. 
 
 Motivated by the theory in \cite{sno-text}, the approach is to choose probabilities every slot to greedily minimize a bound
on the \emph{drift-plus-penalty expression} $\expect{\Delta(t+D) + Vp_0(t)|\bv{Q}(t)}$, where $V$ is a non-negative weight that affects a performance tradeoff. 
The $D$-shifted drift term $\Delta(t+D)$ is different from \cite{sno-text} and is used because of the delayed feedback structure of the queue update \eqref{eq:q-update}.
The intuition is that minimizing $\Delta(t+D)$ maintains queue stability, while adding the weighted penalty term $Vp_0(t)$ biases decisions in favor of lower penalties. 
The following lemma provides a bound on the drift-plus-penalty expression under any $\beta_m(t)$ probabilities. 

\begin{lem} \label{lem:app}  Fix $V\geq 0$. Under the above decision structure, one has for slot $t$: 
\begin{eqnarray}
 \expect{\Delta(t+D) + Vp_0(t)|\bv{Q}(t)} \leq B(1+2D) \nonumber \\
 V\sum_{m=1}^{\tilde{M}}\beta_m(t) r_0^{(m)}  
 + \sum_{k=1}^KQ_k(t)\left[\sum_{m=1}^{\tilde{M}} \beta_m(t) r_k^{(m)}-c_k\right]  \label{eq:dpp} 
 \end{eqnarray}
 where $r_k^{(m)}$ is the $k$th component of $\bv{r}^{(m)}$ as defined in \eqref{eq:rm}, and the constant $B$ is defined: 
 \begin{eqnarray*}
 B \defequiv \max_{m \in \{1, \ldots, \tilde{M}\}} \frac{1}{2}\sum_{k=1}^K\sum_{\bv{\omega}\in\Omega}\pi(\bv{\omega}) \left|\hat{p}_k\left(\bv{g}^{(m)}(\bv{\omega}), \bv{\omega}\right) -c_k\right|^2
 \end{eqnarray*}
\end{lem} 

\begin{proof} 
Note that for all $k \in \{0, 1, \ldots, K\}$: 
\begin{eqnarray*}
\expect{p_k(t)|\bv{Q}(t)} &=& \expect{\hat{p}_k(\bv{\alpha}(t), \bv{\omega}(t))|\bv{Q}(t)} \\
&=& \sum_{m=1}^{\tilde{M}} \sum_{\bv{\omega}\in\Omega} \beta_m(t)\pi(\bv{\omega})\hat{p}_k\left(\bv{g}^{(m)}(\bv{\omega}), \bv{\omega}\right) \\
&=& \sum_{m=1}^{\tilde{M}} \beta_m(t) r_k^{(m)} 
\end{eqnarray*}
Therefore, to prove \eqref{eq:dpp} it suffices to prove: 
\begin{eqnarray}
 \expect{\Delta(t+D) |\bv{Q}(t)} \leq B(1+2D)  \nonumber \\
 + \sum_{k=1}^KQ_k(t)\expect{p_k(t) - c_k|\bv{Q}(t)}  \label{eq:suffices} 
 \end{eqnarray}
 To this end, squaring the queue equation \eqref{eq:q-update}, using $\max[a,0]^2 \leq a^2$, and evaluating at time $t+D$ yields: 
 \begin{eqnarray*}
 Q_k(t+D+1)^2 &\leq& Q_k(t+D)^2 + (p_k(t) - c_k)^2  \nonumber \\
 && + 2Q_k(t+D)(p_k(t)-c_k) 
 \end{eqnarray*}
 Summing over $k \in \{1, \ldots, K\}$ and dividing by $2$ gives: 
 \begin{eqnarray*}
 \Delta(t+D) &\leq& \frac{1}{2}\sum_{k=1}^K(p_k(t)-c_k)^2 \nonumber \\
 && + \sum_{k=1}^KQ_k(t+D)(p_k(t)-c_k)  \\
 &=& \frac{1}{2}\sum_{k=1}^K(p_k(t)-c_k)^2 \nonumber \\
 && + \sum_{k=1}^KQ_k(t)(p_k(t)-c_k)  \\
 && + \sum_{k=1}^K(Q_k(t+D)-Q_k(t))(p_k(t)-c_k)
 \end{eqnarray*}
 Taking conditional expectations of the above proves \eqref{eq:suffices}  upon application of the following inequalities (see Appendix E): 
 \begin{eqnarray*}
 \frac{1}{2}\sum_{k=1}^K\expect{(p_k(t)-c_k)^2|\bv{Q}(t)} \leq B \\
 \sum_{k=1}^K\expect{ (Q_k(t+D)-Q_k(t))(p_k(t)-c_k) | \bv{Q}(t)} \leq 2BD \\
 \end{eqnarray*}
\end{proof} 

\subsection{The drift-plus-penalty algorithm} \label{section:dpp} 

Observe that the probability mass function $\beta_m(t)$ that minimizes the right-hand-side of   \eqref{eq:dpp} is the one that, with probability 1, 
chooses the 
index $m \in \{1, \ldots, \tilde{M}\}$ that minimizes the expression (breaking ties arbitrarily):  
\begin{equation} \label{eq:expression} 
 Vr_0^{(m)} + \sum_{k=1}^KQ_k(t)r_k^{(m)} 
 \end{equation} 
 This gives rise to the following \emph{drift-plus-penalty algorithm}: Every slot $t$: 
 \begin{itemize}
 \item Users observe the queue vector $\bv{Q}(t)$.  
 \item Users select the pure decision strategy $\bv{g}^{(m)}(\bv{\omega})$, where $m$ is the index that minimizes the expression 
 \eqref{eq:expression}. 
 \item The delayed penalty information $p_k(t-D)$ is observed and queues are updated via \eqref{eq:q-update}. 
 \end{itemize}
 
 \subsection{Performance Analysis} 
 
 \begin{thm} \label{thm:performance}  If the problem \eqref{eq:q1}-\eqref{eq:q3} is feasible, then under the drift-plus-penalty algorithm for any $V\geq 0$: 
 \begin{itemize} 
 \item All desired constraints \eqref{eq:q2}-\eqref{eq:q3} are satisfied. 
 \item For all $t>0$, the time average expectation of $p_0(t)$ satisfies: 
 \begin{equation} \label{eq:perf1}
 \frac{1}{t}\sum_{\tau=0}^{t-1}\expect{p_0(\tau)} \leq p_0^{opt} + \frac{B(1+2D)}{V} + \frac{\expect{L(D)}}{Vt}   
 \end{equation} 
 \item For all $t>0$, the time average expectation of $p_k(t)$ satisfies the following for all $k \in \{1, \ldots, K\}$: 
 \begin{equation} \label{eq:bound} 
  \frac{1}{t}\sum_{\tau=0}^{t-1} \expect{p_k(\tau)} \leq c_k + O(\sqrt{V/t}) 
  \end{equation} 
 \end{itemize} 
 \end{thm} 
 
 The above theorem shows the time average expectation of $p_0(t)$ is within $O(1/V)$ of optimality.  It 
 can be pushed as close to optimal as desired
 by increasing the $V$ parameter.  The tradeoff is in the amount of time required for the time average expected penalties to be
 close to their desired constraints.  It can be shown that if $D=0$ and a mild \emph{Slater condition} is satisfied,  then the bound 
 \eqref{eq:bound} can be improved to (see Appendix D): 
  \begin{equation}
  \frac{1}{t}\sum_{\tau=0}^{t-1} \expect{p_k(\tau)} \leq c_k + O(V/t) + O(\log(t)/t) \label{eq:slater} 
  \end{equation} 
  
  \begin{proof} (Theorem \ref{thm:performance}) 
  Every slot $\tau \in \{0, 1, 2, \ldots\}$ the drift-plus-penalty algorithm chooses probabilities $\beta_m(\tau)$ that minimize the right-hand-side of the expression
  \eqref{eq:dpp}.  Thus: 
  \begin{eqnarray*}
 \expect{\Delta(\tau+D) + Vp_0(\tau)|\bv{Q}(\tau)} \leq B(1+2D) \nonumber \\
 V\sum_{m=1}^{\tilde{M}}\theta_m r_0^{(m)}  
 + \sum_{k=1}^KQ_k(\tau)\left[\sum_{m=1}^{\tilde{M}} \theta_m r_k^{(m)}-c_k\right]  
 \end{eqnarray*}
 where $\theta_m$ is any alternative probability mass function defined over $m \in \{1, \ldots, \tilde{M}\}$.  Using the probabilities $\theta_m$ that optimally 
 solve the linear program \eqref{eq:lp1}-\eqref{eq:lp4} gives: 
  \begin{eqnarray*}
 \expect{\Delta(\tau+D) + Vp_0(\tau)|\bv{Q}(\tau)} \leq B(1+2D) +  Vp_0^{opt}   
 \end{eqnarray*}
 Taking expectations of both sides and using iterated expectations gives: 
 \[ \expect{\Delta(\tau+D)} + V\expect{p_0(\tau)} \leq B(1+2D) +  Vp_0^{opt}    \]
 Summing over $\tau \in \{0, 1, \ldots, t-1\}$ gives: 
 \begin{eqnarray}
 \expect{L(t+D)} - \expect{L(D)} +V\sum_{\tau=0}^{t-1} \expect{p_0(\tau)} \leq \nonumber \\
 B(1+2D)t + Vp_0^{opt}t  \label{eq:foo} 
 \end{eqnarray}
 Using the fact that $\expect{L(t+D)} \geq 0$ and rearranging terms proves \eqref{eq:perf1}. 
 
 Again rearranging \eqref{eq:foo} yields: 
 \begin{eqnarray}
 \expect{L(t+D)} \leq (C + FV)t \label{eq:baz} 
 \end{eqnarray}
 where $C$ is defined: 
 \[ C \defequiv \expect{L(D)} + B(1+2D) \]
 and $F$ is defined as a constant that satisfies the following for all slots $\tau$: 
 \[ F \geq p_0^{opt} - \expect{p_0(\tau)} \]
 Such a constant exists because $p_0(\tau)$ has a finite number of possible outcomes. 
 Using the definition of $L(t+D)$ in \eqref{eq:baz}  gives: 
 \[ \expect{\norm{\bv{Q}(t+D)}^2} \leq 2(C+FV)t \]
By Jensen's inequality: 
\[ \expect{\norm{\bv{Q}(t+D)}}^2 \leq 2(C+FV)t \]
Thus: 
\[ \frac{\expect{\norm{\bv{Q}(t+D)}}}{t} \leq \sqrt{\frac{2(C+FV)}{t}} \]
Using this with Lemma \ref{lem:vq} proves \eqref{eq:bound}.  The inequality \eqref{eq:bound} immediately implies that all
desired constraints are satisfied. 
  \end{proof}

\subsection{The approximate drift-plus-penalty algorithm} \label{section:approximate} 

The algorithm of Section \ref{section:dpp} assumes perfect knowledge of the $r_k^{(m)}$ values.  These can be computed by \eqref{eq:rm} if the event 
probabilities $\pi(\bv{\omega})$ are known.  Suppose these probabilities are unknown, but delayed samples $\bv{\omega}(t-D)$ 
are available at the end of each slot $t$.   Let $W$ be a positive integer that represents a \emph{sample size}.  The $r_k^{(m)}$ values can 
be approximated by: 
\[ \tilde{r}_k^{(m)}(t) = \frac{1}{W} \sum_{w=0}^{W-1} \hat{p}_k\left(\bv{g}^{(m)}(\bv{\omega}(t-D-w)), \bv{\omega}(t-D-w)\right)\]
The approximate algorithm uses $\tilde{r}_k^{(m)}(t)$ values in replace of $r_k^{(m)}$ in the expression \eqref{eq:expression}.  Analysis in  \cite{neely-mwl-tac} shows that the performance gap between exact and approximate drift-plus-penalty implementations is 
$O(1/\sqrt{W})$, so that the approximate algorithm is very close to the exact algorithm when $W$ is large.  


\subsection{Separable penalty functions} 

A simpler and exact implementation is possible, without requiring knowledge of the probability distribution for $\bv{\omega}(t)$, 
when penalty functions have the following separable form for all $k \in \{0, 1, \ldots, K\}$: 
\begin{equation} \label{eq:separable} 
\hat{p}_k(\bv{\alpha}, \bv{\omega}) = \sum_{i=1}^N\hat{p}_{ik}(\alpha_i, \omega_i) 
\end{equation} 
where $\hat{p}_{ik}(\alpha_i, \omega_i)$ are any functions of $(\alpha_i, \omega_i) \in \script{A}_i \times \Omega_i$. 
Choosing an $m \in \{1, \ldots, \tilde{M}\}$ that minimizes the expression \eqref{eq:expression} is equivalent to observing the queues $\bv{Q}(t)$ and then 
 choosing a strategy function $\bv{g}(\bv{\omega})= (g_1(\omega_1), \ldots, g_N(\omega_N))$ to minimize: 
\begin{equation*} 
\sum_{\bv{\omega} \in \Omega}\pi(\bv{\omega}) \left[V\hat{p}_0(\bv{g}(\bv{\omega}), \bv{\omega}) + \sum_{k=1}^KQ_k(t)\hat{p}_k(\bv{g}(\bv{\omega}), \bv{\omega})\right] 
\end{equation*} 
With the structure \eqref{eq:separable}, this expression becomes: 
\[ \sum_{\bv{\omega} \in \Omega}\sum_{i=1}^K \pi(\bv{\omega})\left[V\hat{p}_{i0}(g_i(\omega_i), \omega_i) + \sum_{k=1}^KQ_k(t)\hat{p}_{ik}(g_i(\omega_i), \omega_i)\right]  \]
The above is minimized by the following for each $i \in \{1, \ldots, N\}$: 
\[ g_i(\omega_i) = \arg \min_{\alpha_i \in \script{A}_i} \left[V\hat{p}_{i0}(\alpha_i, \omega_i) + \sum_{k=1}^KQ_k(t) \hat{p}_{ik}(\alpha_i, \omega_i)\right] \]
Thus, the minimization step in the drift-plus-penalty algorithm reduces to having each user observe its own $\omega_i(t)$ value and then setting 
$\alpha_i(t) = g_i(\omega_i(t))$, where the function $g_i(\omega_i)$ is defined above.  The queue update \eqref{eq:q-update} is the same as before. 

In the special case $D=0$, this is the same algorithm as the optimal (centralized) drift-plus-penalty algorithm of \cite{sno-text}.  Hence, for separable problems, there is no optimality gap between centralized and distributed algorithms.

\section{Simulations} 

\subsection{Ergodic performance for a 2 user system} 

This subsection presents simulation results for the 2 user sensor network example of Section \ref{section:sensor-network}.
The approximate drift-plus-penalty algorithm of Section \ref{section:approximate} is used with a delay of $D=10$ slots and a moving average window size of $W=40$ slots.  The algorithm is not aware of the system probabilities.  The objective of this simulation is to find how close the achieved utility is to the optimal value $u^{opt} = 23/48 \approx 0.47917$ computed in Section \ref{section:correlated-reports}.   
Recall that the desired power constraints are $\overline{p}_i \leq 1/3$ for each user $i \in \{1, 2\}$. 
The table in Fig. \ref{fig:table} presents performance for various values of $V$.  For $V\geq 50$ the achieved utility differs from optimality only in the fourth decimal place. 

\begin{figure} [cht]
\centering
\begin{tabular}{|c | c | c | c |}
\hline
 $V$&$\overline{u}$&$\overline{p}_1$&$\overline{p}_2$ \\ \hline 
 1 & 0.344639 & 0.259764  &  0.219525 \\ \hline
  5 & 0.454557 &0.333158  & 0.267161 \\ \hline
 10 & 0.472763 & 0.333335 & 0.300415 \\ \hline
 25 & 0.478186 & 0.333346 & 0.326948 \\ \hline
 50 & 0.479032  & 0.333369  & 0.332873 \\ \hline
 100 & 0.479218 & 0.333406  & 0.333334 \\ \hline
 \end{tabular}
 \caption{Algorithm performance over $t=10^6$ slots ($D=10$, $W=40$). Recall that $u^{opt} = 23/48 \approx 0.47917$.}
 \label{fig:table}
\end{figure}

\subsection{Ergodic performance for a 3 user system}

Consider a network of 3 sensors that communicate reports to a fusion center, similar to the example considered in 
Section \ref{section:sensor-network}. The event processes $\omega_i(t)$ for each 
sensor $i \in \{1, 2, 3\}$ take values in the same 10 element set $\Omega$: 
\[ \Omega \defequiv \{0, 1, 2, 3,  \ldots, 9\} \] 
Consider binary actions $\alpha_i(t) \in \{0, 1\}$, where $\alpha_i(t)=1$ corresponds to sensor $i$ sending a report, and incurs a power cost of $1$ for that sensor.  The penalty and utility functions are: 
\begin{eqnarray*}
\hat{p}_i(\alpha_i, \omega_i) &=&  \alpha_i \: \: \: \: \forall i \in \{1, 2, 3\} \\
\hat{u}(\bv{\alpha}, \bv{\omega}) &=& \min\left[\frac{\alpha_1\omega_1}{10} + \frac{\alpha_2\omega_2 + \alpha_3\omega_3}{20}, 1\right] 
\end{eqnarray*}
Thus, sensor 1 brings more utility than the other sensors. 

Assume $\omega_1(t), \omega_2(t), \omega_3(t)$ are mutually independent and uniformly distributed over $\Omega$.  
The requirements for 
Theorem \ref{thm:non-decreasing} hold, and so one can restrict attention to the $11$ threshold functions $g_i(\omega_i)$ 
of the type \eqref{eq:threshold-function}.  As it does not make sense to report when $\omega_i(t)=0$, the functions $g_i(\omega)=1$ for all $\omega$  can be removed.  This leaves only 10 threshold functions at each user, for a total of $10^3=1000$ strategy functions $\bv{g}^{(m)}(\bv{\omega})$ to be considered every slot.  The approximate drift-plus-penalty algorithm of Section \ref{section:approximate} 
is simulated over $t=10^6$ slots with a delay $D=10$ and for various choices of the moving average window size $W$ and the parameter $V$.  All average power constraints were met for all choices of $V$ and $W$.  The achieved utility is shown in Fig. \ref{fig:achieved-utility}.  The utility increases to a limiting value as $V$ is increased.  This limiting value can be improved by adjusting the number of samples $W$ used in the moving average.  Increasing $W$ from $40$ to $200$ gives a small
improvement in performance. There is only a negligible improvement when $W$ is further increased to $400$ (the curves for $W=200$ and $W=400$ look identical).  

\begin{figure}[htbp]
   \centering
   \includegraphics[width=3.5in]{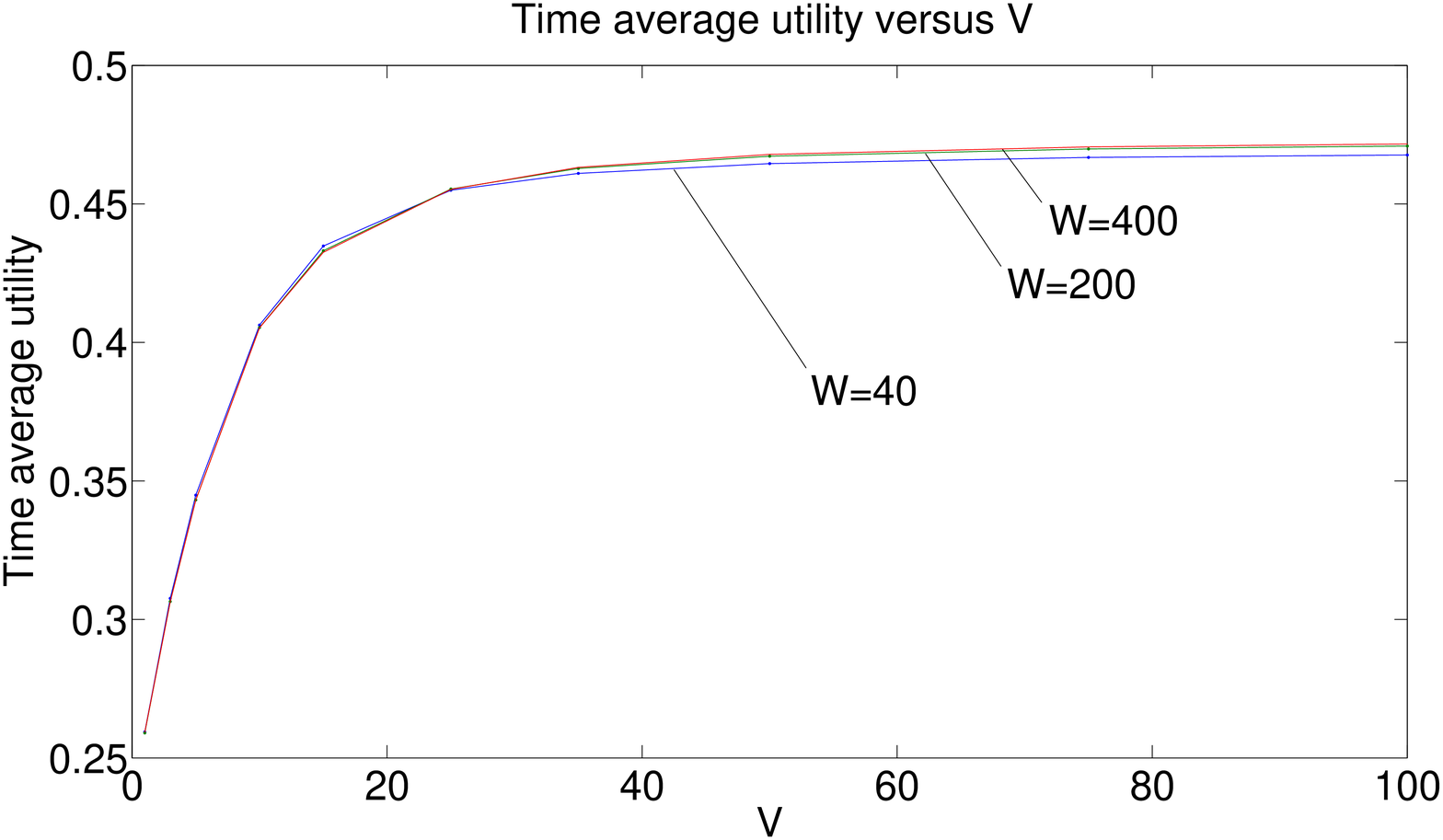} 
   \caption{Achieved utility $\overline{u}$ versus $V$ for various choices of $W$.}
   \label{fig:achieved-utility}
\end{figure}

Fig. \ref{fig:samplepath} demonstrates how the $V$ parameter affects the rate of convergence to the desired constraints. 
The window size is fixed to $W=40$ and the value $\max[\overline{p}_1(t), \overline{p}_2(t), \overline{p}_3(t)]$ is plotted 
for $t \in \{0, 1, \ldots, 2000\}$ (where $\overline{p}_i(t)$ is the empirical average power expenditure of user $i$ up to slot $t$). 
This value approaches the desired constraint of $1/3$ more slowly when $V$ is large.  The following table presents time averages after a longer duration of $10^6$ slots. 

\begin{figure} [htbp]
\centering
\begin{tabular}{|c | c | c | c |c|}
\hline
 $V$&$\overline{u}$&$\overline{p}_1$&$\overline{p}_2$ & $\overline{p}_3$ \\ \hline 
 1 &  0.259400 & 0.258000 & 0.251310 & 0.251342  \\ \hline
  10 & 0.406263   & 0.333301 & 0.316371 & 0.316418 \\ \hline
 50 & 0.464545  & 0.333357  & 0.333341  & 0.333342 \\ \hline
 100 &  0.467642 & 0.333387 & 0.333354 & 0.333354 \\ \hline
  \end{tabular}
 \caption{Time averages after $t=10^6$ slots ($W=40$).}
 \label{fig:table2}
\end{figure}

\begin{figure}[htbp]
   \centering
   \includegraphics[width=3.5in]{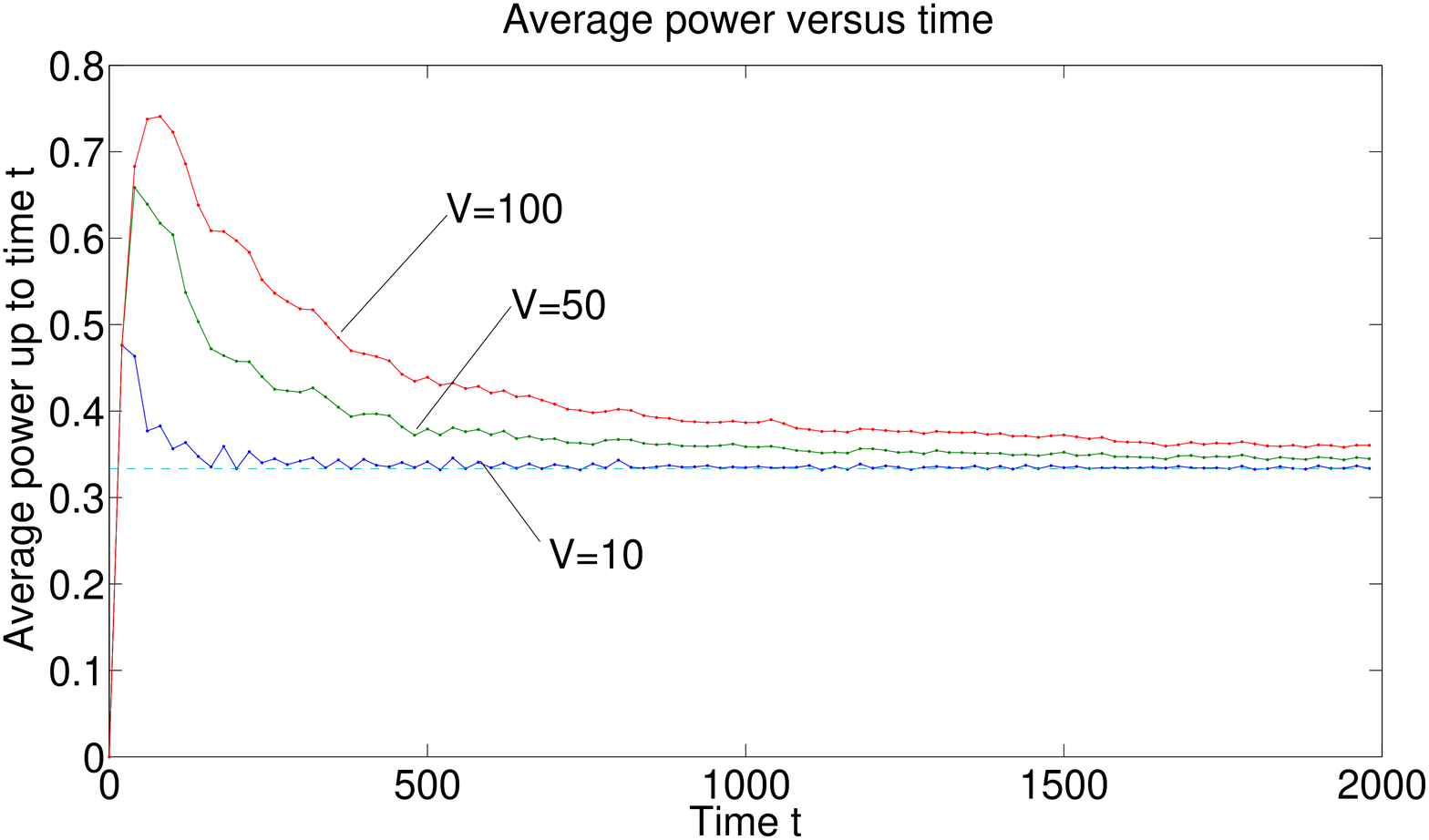} 
   \caption{An illustration of the rate of convergence to the desired constraint $1/3$ for various choices of $V$. The 
   curves plot $\max[\overline{p}_1(t), \overline{p}_2(t), \overline{p}_3(t)]$ versus $t$.}
   \label{fig:samplepath}
\end{figure}


\subsection{Adaptation to non-ergodic changes} 

The initial queue state determines the coefficient of an $O(1/t)$ transient in the performance bounds
of the system (consider the $\expect{L(D)}/(Vt)$ term in  \eqref{eq:perf1}).  Thus, if system probabilities change abruptly, 
the system can be viewed as restarting with a different initial condition.   Thus, one expects the system to react robustly to such changes. 

To illustrate this, consider the same 3-user system of the previous subsection, using $V=50, W=40$.  
The event processes $\omega_i(t)$ have the same probabilities as given in the previous subsection for slots $t<4000$ and $t >8000$.  Call this \emph{distribution type 1}.  However, for slots $t \in \{4000, \ldots, 8000\}$, the $\omega_i(t)$ processes are independently chosen with a different distribution as follows: 
\begin{itemize} 
\item $Pr[\omega_1(t)=0]=Pr[\omega_1(t)=9]=1/2$. 
\item $Pr[\omega_2(t)=k]=1/4$ for $k \in \{6, 7, 8, 9\}$. 
\item $Pr[\omega_3(t)=k]=1/4$ for $k \in \{6, 7, 8, 9\}$. 
\end{itemize} 
This is called \emph{distribution type 2}. 

Fig. \ref{fig:adapt} shows average utility and average power over the first $12000$ slots.  Values at each slot $t$ are averaged over $2000$ independent system runs. The two dashed horizontal lines in the top plot of the figure  are long term time average utilities achieved over  $10^6$ slots under probabilities that are fixed at distribution type 1 and type 2, respectively.  
 It is seen that the system adapts to the non-ergodic change by quickly adjusting to the new optimal average utility. The figure also plots average power of user $1$ versus time, with a dashed horizontal line at the power constraint $1/3$.  A noticeable disturbance in average power occurs at the non-ergodic changes in distribution.

It was observed that system performance is not very sensitive to inaccurate estimates of the $r_k^{(m)}$ values (results not shown in the figures).  This suggests that, for this example, the virtual queues alone are sufficient to ensure the average power constraints are met, which, together with loose estimates for $r_k^{(m)}$, are sufficient to provide an accurate approximation to optimality. 

\begin{figure}[htbp]
   \centering
   \includegraphics[width=3.5in]{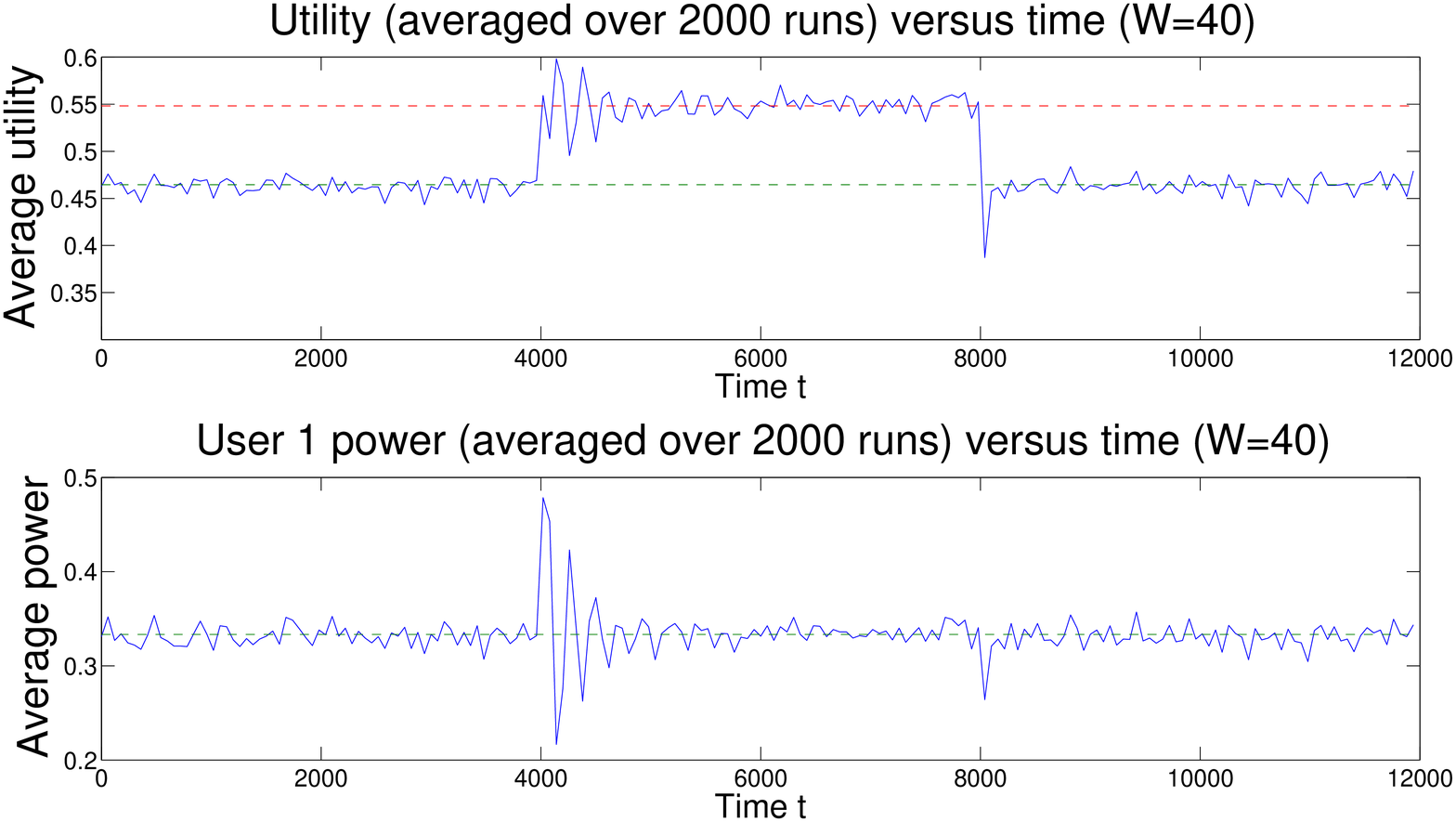} 
   \caption{A sample path of average utility and power versus time.  Values at each time slot $t$ are obtained by averaging the actual utility and power used by the algorithm on that slot over 2000 independent simulation runs.}
   \label{fig:adapt}
\end{figure}

\section{Conclusions} 

This paper treated distributed scheduling in a multi-user system where users know their own observations and actions, but not those of others.  In this context, there is a fundamental performance gap between distributed and centralized decisions.  
Optimal distributed policies were constructed by correlating decisions via a source of common randomness.   
The optimal policy is computable via a linear program if all system probabilities are known, and through an online algorithm with virtual queues if probabilities are unknown.  The online algorithm 
assumes there is delayed feedback about previous penalties and rewards.   The algorithm was shown in simulation to adapt when system probabilities change.  
In the special case when the events observed at each user are independent and when 
penalty and utility functions satisfy a \emph{preferred action property}, the number of pure strategies for consideration on each slot can be significantly reduced.  In some cases,  this reduces an exponentially complex algorithm to one that has only polynomial complexity. 

\section*{Appendix A --- Proof of Theorem \ref{thm:optimality}}

This appendix proves Theorem \ref{thm:optimality}. 
 Define the $(K+1)$-dimensional \emph{penalty vectors}: 
\begin{eqnarray*}
 \bv{p}(t) &=& (p_0(t), p_1(t), \ldots, p_K(t)) \\
 \hat{\bv{p}}(\bv{\alpha}, \bv{\omega}) &=& (\hat{p}_0(\bv{\alpha}, \bv{\omega}), \hat{p}_1(\bv{\alpha}, \bv{\omega}), \ldots, \hat{p}_K(\bv{\alpha}, \bv{\omega}))
 \end{eqnarray*} 
For each $m \in \{1, \ldots, M\}$, 
define: 
\[ \bv{r}^{(m)} \defequiv \sum_{\bv{\omega}\in\Omega} \pi(\bv{\omega})\hat{\bv{p}}(\bv{g}^{(m)}(\bv{\omega}), \bv{\omega}) = (r_0^{(m)}, r_1^{(m)}, \ldots, r_K^{(m)})  \]
Define $\script{R}$ as the convex hull of these vectors: 
\[ \script{R} \defequiv Conv\left(\{\bv{r}^{(1)}, \ldots, \bv{r}^{(M)}\}\right) \]
The set $\script{R}$ is convex, closed, and bounded.  From the nature of the convex hull operation, the set $\script{R}$ can be viewed as the set of all average penalty 
vectors achievable by timesharing over the $M$ different 
pure strategies. 

\begin{lem}  \label{lem:sat} Let $\bv{\alpha}(t)$ be decisions of an algorithm that satisfies the distributed scheduling constraint \eqref{eq:distributed-constraint} on every slot. Then: 

(a) For all slots $t \in \{0, 1, 2, \ldots\}$: 
\[ \expect{\bv{p}(t)} \in \script{R} \]

(b) For all slots $t\in \{1, 2, 3, \ldots\}$: 
\[ \overline{\bv{p}}(t) \in \script{R} \]
where
\[ \overline{\bv{p}}(t) \defequiv \frac{1}{t}\sum_{\tau=0}^{t-1} \expect{\bv{p}(\tau)} \]
\end{lem} 

\begin{proof} 
Part (b) follows immediately from part (a) together with the fact that $\script{R}$ is convex. 
To prove part (a), fix a slot $t \in \{0, 1, 2, \ldots\}$.  By \eqref{eq:distributed-constraint}, the users make decisions: 
\[ \bv{\alpha}(t) = \left(f_{1}(\omega_1(t), X(t)), \ldots, f_{N}(\omega_N(t), X(t))\right) \]
For each $X(t) \in \script{X}$ and $\bv{\omega} \in \Omega$, define: 
\[ \bv{g}_{X(t)}(\bv{\omega}) = \left(f_1(\omega_1, X(t)), \ldots, f_N(\omega_N, X(t))\right)  \]
Then, given $X(t)$, the function $\bv{g}_{X(t)}(\bv{\omega})$ is a pure strategy.  Hence, 
$\bv{g}_{X(t)}(\bv{\omega}) = \bv{g}^{(m)}(\bv{\omega})$ for some $m \in \{1, \ldots, M\}$.  Define $m_{X(t)}$ as the value $m \in \{1, \ldots, M\}$ 
for which this holds.  
Thus, $\bv{g}_{X(t)}(\bv{\omega}) = \bv{g}^{(m_{X(t)})}(\bv{\omega})$, and: 
\begin{eqnarray*}
 \expect{\bv{p}(t)|X(t)} &=& \expect{\hat{\bv{p}}(\bv{\alpha}(t), \bv{\omega}(t))|X(t)} \\ 
 &=& \expect{\hat{\bv{p}}\left(\bv{g}^{(m_{X(t)})}(\bv{\omega}(t)), \bv{\omega}(t)\right)|X(t)} \\
 &=& \sum_{\bv{\omega}\in\Omega} \pi(\bv{\omega}) \hat{\bv{p}}\left(\bv{g}^{(m_{X(t)})}(\bv{\omega}), \bv{\omega}\right) \\
 &=& \bv{r}^{(m_{X(t)})} 
\end{eqnarray*}
Taking expectations of both sides and using the law of iterated expectations gives: 
\[ \expect{\bv{p}(t)} = \sum_{m=1}^MPr[m_{X(t)}=m]\bv{r}^{(m)} \]
The above is a convex combination of $\{\bv{r}^{(1)}, \ldots, \bv{r}^{(M)}\}$, and hence is in $\script{R}$. 
\end{proof} 


\begin{lem} \label{lem:sat2} There exist real numbers $r_1, r_2, \ldots, r_K$ that satisfy the following: 
\begin{eqnarray}
r_k \leq c_k \: \: \: \: \forall k \in \{1, \ldots, K\} \label{eq:sat1} \\
(p_0^{opt}, r_1, r_2, \ldots, r_K) \in \script{R} \label{eq:sat2} 
\end{eqnarray} 
Furthermore, the vector in \eqref{eq:sat2} is on the \emph{boundary} of $\script{R}$. 
\end{lem} 

\begin{proof} 
Fix $q$ as a positive integer.   Consider an algorithm that satisfies the distributed scheduling constraint \eqref{eq:distributed-constraint} every slot.
For $k \in \{0, 1, \ldots, K\}$, let $\overline{p}_k(t)$ be the resulting time average expected penalties.  
Assume the algorithm satisfies: 
\begin{eqnarray}
p_0^{opt} \leq \limsup_{t\rightarrow\infty} \overline{p}_0(t) &\leq& p_0^{opt} + 1/q \label{eq:proof1} \\
\limsup_{t\rightarrow\infty} \overline{p}_k(t) &\leq& c_k \: \: \: \: \forall k \in \{1, \ldots, K\} \label{eq:proof2} 
 \end{eqnarray}
Such an algorithm must exist because $p_0^{opt}$ is the infimum objective value for \eqref{eq:q1} over all algorithms that satisfy the constraints \eqref{eq:q2}-\eqref{eq:q3}.

Lemma \ref{lem:sat} implies that $\overline{\bv{p}}(t) = (\overline{p}_0(t), \ldots, \overline{p}_K(t)) \in \script{R}$ for all $t>0$. 
Let $t_n$ be a subsequence of times over which $\overline{p}_0(t)$ achieves its $\limsup$.   Since $\overline{\bv{p}}(t_n)$ is in the closed and bounded
set  $\script{R}$ for all $t_n>0$, 
the Bolzano-Wierstrass theorem implies there is a subsequence $\overline{\bv{p}}(t_{n_m})$ that converges to a point
$\bv{r}(q) \in \script{R}$, where $\bv{r}(q) = (r_0(q), \ldots, r_K(q))$.  Thus: 
\begin{eqnarray} 
r_0(q) &=& \lim_{m\rightarrow\infty} \overline{p}_0(t_{n_m}) = \limsup_{t\rightarrow\infty} \overline{p}_0(t)  \label{eq:mid-proof} \\
r_k(q) &=& \lim_{m\rightarrow\infty} \overline{p}_k(t_{n_m}) \leq \limsup_{t\rightarrow\infty}\overline{p}_k(t)  \: \: \forall k \in \{1, \ldots, K\}\nonumber 
\end{eqnarray} 
Using \eqref{eq:proof2} in the last inequality above gives: 
\begin{equation} \label{eq:proof4} 
r_k(q) \leq c_k \: \: \: \: \forall k \in \{1, \ldots, K\} 
\end{equation} 
Further, substituting \eqref{eq:mid-proof} into
\eqref{eq:proof1} gives: 
\begin{eqnarray}
p_0^{opt} \leq r_0(q) \leq p_0^{opt} + 1/q \label{eq:proof5} 
\end{eqnarray}

This holds for all positive integers $q$.  Thus,  $\{\bv{r}(q)\}_{q=1}^{\infty}$ is an infinite sequence of vectors in  $\script{R}$ such that $\bv{r}(q)$ satisfies
\eqref{eq:proof4} and \eqref{eq:proof5} for all $q \in \{1, 2, 3, \ldots\}$.   Because $\script{R}$ is closed and bounded, the sequence $\{\bv{r}(q)\}_{q=1}^{\infty}$ has
a limit point 
$\bv{r} = (r_0, r_1, \ldots, r_K) \in \script{R}$ that satisfies $r_0=p_0^{opt}$ and 
$r_k \leq c_k$ for all $k \in \{1, \ldots, K\}$.  This proves \eqref{eq:sat1} and \eqref{eq:sat2}. 

To prove that $\bv{r}$ is on the \emph{boundary} of $\script{R}$, it suffices to note that for any $\epsilon>0$: 
\[ (p_0^{opt} - \epsilon, r_1, \ldots, r_K) \notin \script{R} \]
Indeed, if this were not true, it would be possible to construct a distributed algorithm that satisfies all desired constraints and yields
a time average expected value of $p_0(t)$ equal to $p_0^{opt} - \epsilon$, which contradicts the definition of $p_0^{opt}$. 
\end{proof} 

Because $\script{R}=Conv(\{\bv{r}^{(1)}, \ldots, \bv{r}^{(M)}\})$, 
Lemma \ref{lem:sat2} implies there are probabilities $\theta_m$ that sum to 1 such that: 
\[ (p_0^{opt}, r_1, \ldots, r_K) = \sum_{m=1}^M\theta_m \bv{r}^{(m)} \]
Because $\script{R}$ is a $(K+1)$-dimensional set, 
Caratheodory's theorem ensures the above can be written using at most $K+2$ non-zero $\theta_m$ values.  However, 
because the above vector is on the \emph{boundary} of $\script{R}$, a simple extension of Caratheodory's theorem ensures it can be written
using at most $K+1$ non-zero $\theta_m$ values.\footnote{This extension to points on the boundary of a convex hull can be proven using Caratheodory's theorem together with the 
supporting hyperplane theorem for convex sets \cite{bertsekas-convex}.} 
  This proves Theorem \ref{thm:optimality}. 

\section*{Appendix B --- A counterexample}

This appendix shows it is possible for an algorithm to satisfy the conditional independence assumption \eqref{eq:bad-distributed-constraint} while yielding expected utility strictly larger than that of any distributed algorithm.  Consider a two user system with $\omega_1(t), \omega_2(t)$ independent and i.i.d. Bernoulli processes with: 
\[ Pr[\omega_i(t)=1] = Pr[\omega_i(t)=0] = 1/2 \: \: \forall i \in \{1, 2\} \]
The actions are constrained to: 
\[ \alpha_1(t) \in \{-1, 1\} \: \: \: \: , \: \: \: \:   \alpha_2(t) \in \{-1, 1\}  \]
Define the utility function:
\[ \hat{u}(\alpha_1, \alpha_2, \omega_1, \omega_2) = g(\omega_1, \omega_2)\alpha_1\alpha_2 \]
where $g(\omega_1, \omega_2) = 1 - 2\omega_1\omega_2$.  Then $\hat{u}(\cdot) \in \{-1, 1\}$.  Fig. \ref{fig:counterexample} indicates when the utility is 1. 

\begin{figure}[cht] 
\centering
\begin{tabular}{|c | c | c | c |}
\hline 
$\omega_1$ & $\omega_2$ & $g(\omega_1, \omega_2)$ & Conditions required for $\hat{u}=1$  \\ \hline
 0 & 0 & 1 & $\alpha_1 = \alpha_2$ \\
 0 & 1 & 1 & $\alpha_1 = \alpha_2$ \\
 1 & 0 & 1 & $\alpha_1 = \alpha_2$ \\ \hline
 1 & 1 & -1 & $\alpha_1 \neq \alpha_2$ \\
  \hline
 \end{tabular}
 \caption{A table showing the conditions needed for $\hat{u}(\alpha_1, \alpha_2, \omega_1, \omega_2)=1$.}
 \label{fig:counterexample}
\end{figure}

Consider now the following \emph{centralized algorithm}:  Every slot $t$, observe $(\omega_1(t), \omega_2(t))$ and compute $g(\omega_1(t), \omega_2(t))$. 
\begin{itemize} 
\item If $g(\omega_1(t), \omega_2(t))=1$, independently choose: 
\[ (\alpha_1(t), \alpha_2(t)) = \left\{ \begin{array}{ll}
(1, 1)  &\mbox{ with probability 1/2} \\
(-1, -1)  & \mbox{ with probability 1/2} 
\end{array}
\right.\]
\item If $g(\omega_1(t), \omega_2(t)) = -1$, independently choose: 
\[ (\alpha_1(t), \alpha_2(t)) = \left\{ \begin{array}{ll}
(1, -1)  &\mbox{ with probability 1/2} \\
(-1, 1)  & \mbox{ with probability 1/2} 
\end{array}
\right.\]
\end{itemize} 
The randomization ensures that regardless of $(\omega_1(t), \omega_2(t))$: 
\begin{eqnarray*}
Pr[\alpha_1(t)=1|\omega_1(t), \omega_2(t)] &=& \frac{1}{2} \\
Pr[\alpha_2(t)=1|\omega_1(t), \omega_2(t)] &=& \frac{1}{2}
\end{eqnarray*}
and hence the conditional independence assumption \eqref{eq:bad-distributed-constraint} is satisfied.  This algorithm guarantees the utility function is 1 for all possible outcomes, and so the expected utility is also 1.  
However, it can be shown that an optimal \emph{distributed algorithm} is the pure strategy $\alpha_1(t)=\alpha_2(t)=1$ for all $t$ (regardless of $\omega_1(t), \omega_2(t)$), which yields an expected utility of only $1/2$. 

\section*{Appendix C --- Preferred Action Lemmas} 

This appendix provides proofs of Lemmas \ref{lem:preferred1}-\ref{lem:preferred4}.  
The proofs of Lemmas \ref{lem:preferred1} and \ref{lem:preferred2} follow from the following lemma. 

\begin{lem} \label{lem:preferred-new} A penalty function $\hat{p}(\bv{\alpha}, \bv{\omega})$ has the preferred action property if it 
satisfies the following three properties: 
\begin{itemize} 
\item $\script{A}_i = \{0,1\}$ for $i \in \{1, \ldots, N\}$. 
\item $\hat{p}(\bv{\alpha}, \bv{\omega})$ is non-increasing in the vector $\bv{\omega}$.  That is, for all $\bv{\alpha} \in \script{A}$ and all vectors $\bv{\omega}, \bv{\gamma} \in \Omega$ that satisfy $\bv{\omega} \leq \bv{\gamma}$ (with inequality taken entrywise), one has
\[ \hat{p}(\bv{\alpha}, \bv{\omega}) \geq \hat{p}(\bv{\alpha}, \bv{\gamma}) \] 
\item Given $\alpha_i=0$, $\hat{p}(\bv{\alpha},\bv{\omega})$ does not depend on $\omega_i$.  That is, for all $i \in \{1, \ldots, N\}$, all possible values of $\bv{\alpha}_{\overline{i}} \in \script{A}_{\overline{i}}$, $\bv{\omega}_{\overline{i}}\in\Omega_{\overline{i}}$, and all $\omega, \gamma \in \Omega_i$, one has: 
\[  \hat{p}([\bv{\alpha}_{\overline{i}}, 0], [\bv{\omega}_{\overline{i}}, \omega])  = \hat{p}([\bv{\alpha}_{\overline{i}}, 0], [\bv{\omega}_{\overline{i}}, \gamma])   \]
\end{itemize} 
\end{lem} 
\begin{proof} 
Fix $i \in \{1, \ldots, N\}$, fix $\bv{\alpha}_{\overline{i}}$, $\bv{\omega}_{\overline{i}}$, and fix $\alpha, \beta \in \{0,1\}$, $\omega, \gamma \in \Omega_i$ that 
satisfy $\alpha > \beta$ and $\omega < \gamma$. Since $\alpha, \beta$ are binary numbers that satisfy $\alpha > \beta$, it must be that $\alpha=1$, $\beta=0$. 
The goal is to show: 
\begin{eqnarray*}
\hat{p}([\bv{\alpha}_{\overline{i}}, 1], [\bv{\omega}_{\overline{i}}, \omega]) - \hat{p}([\bv{\alpha}_{\overline{i}}, 0], [\bv{\omega}_{\overline{i}}, \omega]) \\
\geq 
\hat{p}([\bv{\alpha}_{\overline{i}}, 1], [\bv{\omega}_{\overline{i}}, \gamma]) - \hat{p}([\bv{\alpha}_{\overline{i}}, 0], [\bv{\omega}_{\overline{i}}, \gamma]) 
\end{eqnarray*}
Since the second term on the left-hand-side is the same as the second term on the right-hand-side, it suffices to show: 
\begin{eqnarray*}
\hat{p}([\bv{\alpha}_{\overline{i}}, 1], [\bv{\omega}_{\overline{i}}, \omega]) 
\geq 
\hat{p}([\bv{\alpha}_{\overline{i}}, 1], [\bv{\omega}_{\overline{i}}, \gamma]) 
\end{eqnarray*}
The above inequality is true because $\omega < \gamma$ and $\hat{p}(\bv{\alpha}, \bv{\omega})$ is non-increasing in the vector $\bv{\omega}$. 
\end{proof} 

\begin{proof} (Lemma \ref{lem:preferred1}) 
Suppose: 
\[ \hat{p}(\bv{\alpha}, \bv{\omega}) = -\min\left[\sum_{i=1}^N\phi_i(\omega_i)\alpha_i , b\right] \]
where $\script{A}_i = \{0, 1\}$ for $i \in \{1, \ldots, N\}$, $b$ is a real number, and 
all functions $\phi_i(\omega_i)$ are non-decreasing in $\omega_i$.  Then $\hat{p}(\bv{\alpha}, \bv{\omega})$ is non-increasing
in the $\bv{\omega}$ vector. Furthermore, for any given $i \in \{1, \ldots, N\}$, any $\bv{\alpha}_{\overline{i}} \in \script{A}_{\overline{i}}$, $\bv{\omega}_{\overline{i}}\in\Omega_{\overline{i}}$, and any $\omega, \gamma \in \Omega_i$,  one has: 
\begin{eqnarray*}
 \hat{p}([\bv{\alpha}_{\overline{i}}, 0], [\bv{\omega}_{\overline{i}}, \omega]) &=& -\min\left[\sum_{j\neq i}\phi_j(\omega_j)\alpha_j, b\right] \\
 &=&  \hat{p}([\bv{\alpha}_{\overline{i}}, 0], [\bv{\omega}_{\overline{i}}, \gamma])
\end{eqnarray*}
Thus, $\hat{p}(\bv{\alpha}, \bv{\omega})$ satisfies the requirements of Lemma \ref{lem:preferred-new}. 
\end{proof} 

\begin{proof} (Lemma \ref{lem:preferred2}) Suppose: 
\[  \hat{p}(\bv{\alpha}, \bv{\omega}) = -\sum_{i=1}^N\omega_i\alpha_i\prod_{j \neq i} (1-\alpha_j) \]
where $\alpha_i \in \{0, 1\}$ and $\omega_i \in \{0, 1, \ldots, |\Omega_i|-1\}$ for all $i \in \{1, \ldots, N\}$.  Then 
$\hat{p}(\bv{\alpha}, \bv{\omega})$ is non-increasing in the $\bv{\omega}$ vector.  Now 
fix $i \in \{1, \ldots, N\}$, fix $\bv{\alpha}_{\overline{i}}$, $\bv{\omega}_{\overline{i}}$, and fix $\omega, \gamma \in \Omega_i$.  Then: 
\begin{eqnarray*}
  \hat{p}([\bv{\alpha}_{\overline{i}}, 0], [\bv{\omega}_{\overline{i}}, \omega])  &=& -\sum_{k\neq i} \omega_k\alpha_k\prod_{j\neq k}(1-\alpha_j)\\
  &=&  \hat{p}([\bv{\alpha}_{\overline{i}}, 0], [\bv{\omega}_{\overline{i}}, \gamma]) 
  \end{eqnarray*}
  Thus, $\hat{p}(\bv{\alpha}, \bv{\omega})$ satisfies the requirements of Lemma \ref{lem:preferred-new}. 
\end{proof} 

\begin{proof} (Lemma \ref{lem:preferred3})  
Suppose: 
\[ \hat{p}(\bv{\alpha}, \bv{\omega}) = \prod_{i=1}^N\phi_i(\omega_i)\psi_i(\alpha_i)   \]
where $\phi_i(\omega_i)$ is non-negative and non-increasing in $\omega_i$ and $\psi_i(\alpha_i)$ is non-negative and non-decreasing in $\alpha_i$. 
Fix $i \in \{1, \ldots, N\}$, fix $\bv{\alpha}_{\overline{i}}$, $\bv{\omega}_{\overline{i}}$, and fix $\alpha, \beta \in \script{A}_i$, $\omega, \gamma \in \Omega_i$ that 
satisfy $\alpha > \beta$ and $\omega < \gamma$.  The goal is to show: 
\begin{eqnarray*}
\hat{p}([\bv{\alpha}_{\overline{i}}, \alpha], [\bv{\omega}_{\overline{i}}, \omega]) - \hat{p}([\bv{\alpha}_{\overline{i}}, \beta], [\bv{\omega}_{\overline{i}}, \omega]) \\
\geq 
\hat{p}([\bv{\alpha}_{\overline{i}}, \alpha], [\bv{\omega}_{\overline{i}}, \gamma]) - \hat{p}([\bv{\alpha}_{\overline{i}}, \beta], [\bv{\omega}_{\overline{i}}, \gamma]) 
\end{eqnarray*}
By canceling common (non-negative) factors, it suffices to show: 
\begin{eqnarray*}
\phi_i(\omega)\psi_i(\alpha) - \phi_i(\omega)\psi_i(\beta) \geq \phi_i(\gamma)\psi_i(\alpha) - \phi_i(\gamma)\psi_i(\beta) 
\end{eqnarray*}
This is equivalent to: 
\begin{eqnarray}
\phi_i(\omega)(\psi_i(\alpha) - \psi_i(\beta)) \geq \phi_i(\gamma)(\psi_i(\alpha) - \psi_i(\beta)) \label{eq:suffices2}
\end{eqnarray}
Since $\alpha > \beta$ and $\psi_i(\alpha)$ is non-decreasing, 
one has $\psi_i(\alpha)-\psi_i(\beta) \geq 0$. By canceling the common (non-negative) factor, it suffices to show: 
\[ \phi_i(\omega) \geq \phi_i(\gamma) \]
This is true because $\omega < \gamma$ and $\phi_i(\omega)$ is non-increasing. 
\end{proof} 

\begin{proof} (Lemma \ref{lem:preferred4}) Suppose: 
\[ \hat{p}(\bv{\alpha}, \bv{\omega}) =  \sum_{r=1}^R w_r\hat{p}_r(\bv{\alpha}, \bv{\omega}) \]
where $w_r$ are non-negative constants, and each function $\hat{p}_r(\bv{\alpha}, \bv{\omega})$ has the preferred
action property. Fix $i \in \{1, \ldots, N\}$, fix $\bv{\alpha}_{\overline{i}}$, $\bv{\omega}_{\overline{i}}$, and fix $\alpha, \beta \in \script{A}_i$, $\omega, \gamma \in \Omega_i$ that 
satisfy $\alpha > \beta$ and $\omega < \gamma$.  Since each function $\hat{p}_r(\bv{\alpha}, \bv{\omega})$ has the preferred
action property, one has for all $r \in \{1, \ldots, R\}$: 
\begin{eqnarray*}
\hat{p}_r([\bv{\alpha}_{\overline{i}}, \alpha], [\bv{\omega}_{\overline{i}}, \omega]) - \hat{p}_r([\bv{\alpha}_{\overline{i}}, \beta], [\bv{\omega}_{\overline{i}}, \omega]) \\
\geq 
\hat{p}_r([\bv{\alpha}_{\overline{i}}, \alpha], [\bv{\omega}_{\overline{i}}, \gamma]) - \hat{p}_r([\bv{\alpha}_{\overline{i}}, \beta], [\bv{\omega}_{\overline{i}}, \gamma]) 
\end{eqnarray*}
Multiplying the above inequality by $w_r$ and summing over $r \in \{1, \ldots, R\}$ proves that $\hat{p}(\bv{\alpha}, \bv{\omega})$ has the preferred
action property. 
\end{proof} 

\section{Appendix D --- The Slater condition} 

For a given real number $\epsilon\geq 0$, consider the following linear program that is related to the linear program \eqref{eq:lp1}-\eqref{eq:lp4}: 
\begin{eqnarray}
\mbox{Minimize:} & \sum_{m=1}^M\theta_mr_0^{(m)}  \label{eq:newlp1} \\
\mbox{Subject to:} & \sum_{m=1}^M\theta_m r_k^{(m)} \leq c_k - \epsilon \: \: \forall k \in \{1, \ldots, K\} \label{eq:newlp2} \\
& \theta_m \geq 0 \: \: \forall m \in \{1, \ldots, M\} \label{eq:newlp3} \\
& \sum_{m=1}^M\theta_m = 1 \label{eq:newlp4} 
\end{eqnarray}
If $\epsilon>0$, 
the penalty constraints are tighter above than in the linear program \eqref{eq:lp1}-\eqref{eq:lp4} (compare 
\eqref{eq:newlp2} and \eqref{eq:lp2}). 
Define $G(\epsilon)$ as the the optimal objective value \eqref{eq:newlp1} as a function of the parameter $\epsilon$.  Then $G(0) = p_0^{opt}$, where $p_0^{opt}$ corresponds to the original linear program \eqref{eq:lp1}-\eqref{eq:lp4}.  Define $\epsilon_{max}$ as the largest value of $\epsilon$ for which \eqref{eq:newlp1}-\eqref{eq:newlp4} is feasible.  Suppose $\epsilon_{max}>0$.  This means it is possible to satisfy the desired time average penalty constraints with a slackness of $\epsilon_{max}$ in each constraint $k \in \{1, \ldots, K\}$. The condition $\epsilon_{max} > 0$ is called the \emph{Slater condition} \cite{bertsekas-nonlinear}. 

For simplicity of exposition, assume $D=0$. 
Since the drift-plus-penalty algorithm takes actions that minimize the right-hand-side of \eqref{eq:dpp} over all probability mass functions
$\beta_m(t)$, one has: 
\begin{eqnarray*}
 \expect{\Delta(t) + Vp_0(t)|\bv{Q}(t)} \leq B \nonumber \\
 V\sum_{m=1}^{\tilde{M}}\theta_m r_0^{(m)}  
 + \sum_{k=1}^KQ_k(t)\left[\sum_{m=1}^{\tilde{M}} \theta_m r_k^{(m)}-c_k\right]   
 \end{eqnarray*}
 for any values $\theta_m$ that satisfy \eqref{eq:newlp3}-\eqref{eq:newlp4}.  Using $\theta_m$ values that solve \eqref{eq:newlp1}-\eqref{eq:newlp4} for the case $\epsilon=\epsilon_{max}$ gives: 
\begin{eqnarray*}
 \expect{\Delta(t) + Vp_0(t)|\bv{Q}(t)} \leq B \nonumber \\
 VG(\epsilon_{max})  
 -\epsilon_{max} \sum_{k=1}^KQ_k(t)  
 \end{eqnarray*}
 Therefore, for all slots $t \in \{0, 1, 2, \ldots\}$ one has: 
 \begin{equation} \label{eq:therefore}
  \expect{\Delta(t)|\bv{Q}(t)} \leq B + FV - \epsilon_{max}\sum_{k=1}^KQ_k(t) 
  \end{equation} 
 where $F$ is a constant that satisfies the following for all slots $t$ and all possible values of 
 $\bv{Q}(t)$: 
 \[ F \geq G(\epsilon_{max}) -  \expect{p_0(t)|\bv{Q}(t)}  \]
Now define $\delta_{max}$ as the largest possible change in $\norm{\bv{Q}(t)}$ from one slot to the next, so that regardless of the control decisions, one has: 
\begin{equation} \label{eq:delta-max} 
\left|\norm{\bv{Q}(t+1)} - \norm{\bv{Q}(t)}\right| \leq \delta_{max} \: \: \forall t \in \{0, 1,2, \ldots\} 
\end{equation} 
Such a value $\delta_{max}$ exists because all penalty functions $\hat{p}_k(\bv{\alpha}(t), \bv{\omega}(t))$ are bounded.

\begin{lem} \label{lem:slater} Let $\delta_{max}$ be a positive value that satisfies \eqref{eq:delta-max}. 
Let $A$ be a non-negative real number,  and let $\epsilon>0$.  Assume $\norm{\bv{Q}(0)}=0$ with probability 1, and that 
for all slots $t$ and 
all possible $\bv{Q}(t)$ one has: 
\begin{eqnarray} \label{eq:cond} 
\expect{\Delta(t)|\bv{Q}(t)} \leq A - \epsilon \sum_{k=1}^KQ_k(t)
\end{eqnarray}
Then for all slots $t \in \{1, 2, \ldots\}$: 
\begin{eqnarray*}
 &&\hspace{-.3in}\expect{\norm{\bv{Q}(t)}} \leq \\
 &&\max\left[\frac{\log(2)}{r}, \max\left[\frac{2A}{\epsilon}, \frac{\epsilon}{2}\right] +\frac{\log(2t[e^{r\delta_{max}}-1])}{r}\ \right] 
 \end{eqnarray*}
 where $r$ is defined: 
  \begin{equation} \label{eq:r} 
 r = \frac{\epsilon}{\delta_{max}^2 + \epsilon\delta_{max}/3} 
 \end{equation}  
\end{lem}

Using $A = B + FV$ in \eqref{eq:therefore} shows that the system under study satisfies the requirements of the above lemma, 
which proves that 
\eqref{eq:slater} holds. 
The proof of the above lemma relies heavily on drift analysis in \cite{longbo-lagrange-tac} and results for exponentiated
martingales in \cite{martingale-concentration}. 

\begin{proof} (Lemma \ref{lem:slater}) 
Suppose that:
\begin{equation}
 \norm{\bv{Q}(t)} \geq \max\left[2A/\epsilon, \epsilon/2\right] \label{eq:suppose} 
\end{equation} 
By definition of $\Delta(t)$, one has from \eqref{eq:cond}:
\begin{eqnarray}
&&\hspace{-.3in}\expect{\norm{\bv{Q}(t+1)}^2|\bv{Q}(t)} \nonumber \\
&\leq& \norm{\bv{Q}(t)}^2  + 2A - 2\epsilon \sum_{k=1}^KQ_k(t)  \nonumber \\
&\leq& \norm{\bv{Q}(t)}^2 + 2A - 2\epsilon\norm{\bv{Q}(t)}  \label{eq:g1} \\
&\leq& \norm{\bv{Q}(t)}^2 - \epsilon\norm{\bv{Q}(t)} \label{eq:g2} \\
&\leq& (\norm{\bv{Q}(t)} - \epsilon/2)^2 \nonumber
\end{eqnarray}
where \eqref{eq:g1} holds because the sum of the components of a non-negative vector is greater than or equal to its norm, 
and  \eqref{eq:g2} holds because \eqref{eq:suppose} implies $\epsilon\norm{\bv{Q}(t))} \geq 2A$. By Jensen's inequality: 
\begin{eqnarray*}
\expect{\norm{\bv{Q}(t+1)}|\bv{Q}(t)}^2 
\leq (\norm{\bv{Q}(t)} - \epsilon/2)^2
\end{eqnarray*}
Taking the square root of both sides and using \eqref{eq:suppose} gives: 
\begin{equation} \label{eq:holds} 
 \expect{\norm{\bv{Q}(t+1)|\bv{Q}(t)}} \leq \norm{\bv{Q}(t)} - \epsilon/2 
 \end{equation} 
 Define $C$ by: 
 \[ C \defequiv  \max\left[2A/\epsilon, \epsilon/2\right]  \]
 so that \eqref{eq:holds} holds whenever $\norm{\bv{Q}(t)} \geq C$. 
Define $\delta(t)$ by: 
\begin{eqnarray*}
 \delta(t) \defequiv \norm{\bv{Q}(t+1)} - \norm{\bv{Q}(t)} 
 \end{eqnarray*}
and note that $|\delta(t)| \leq \delta_{max}$ for all $t$. 
It follows that: 
\begin{eqnarray}
\expect{\delta(t)|\bv{Q}(t)}
\leq \left\{ \begin{array}{ll}
  - \epsilon/2 &\mbox{ if $\norm{\bv{Q}(t)} \geq C$}  \\
\delta_{max}  & \mbox{ otherwise}  
\end{array}
\right.
\label{eq:bracket} 
\end{eqnarray}

Define $Y(t) = e^{r\norm{\bv{Q}(t}}$ for a positive value of $r$ to be determined.  Assume that $r$ satisfies: 
\begin{equation} \label{eq:rcond1} 
 0 \leq r\delta_{max} < 3 
 \end{equation} 
  Then: 
\begin{eqnarray*}
Y(t+1) - Y(t) &=& e^{r\norm{\bv{Q}(t)}} e^{r\delta(t)}- Y(t) \\
 &=& Y(t) [e^{r\delta(t)}-1]  \\
 &\leq& \left\{ \begin{array}{ll}
  Y(t)[e^{r\delta(t)} - 1] &\mbox{ if $\norm{\bv{Q}(t)} \geq C$}  \\
e^{rC}[e^{r\delta_{max}} - 1]  & \mbox{ otherwise}  
\end{array}
\right.
\end{eqnarray*}

Now define $g(x)$ as the function that satisfies the following for all real numbers $x$: 
\begin{equation} \label{eq:expansion} 
 e^{x} -1 =  x + \frac{x^2}{2} g(x) 
 \end{equation} 
By results in \cite{martingale-concentration}, the function $g(x)$ is non-decreasing in $x$ and satisfies:  
\begin{equation} \label{eq:g} 
g(x) \leq \frac{1}{1-x/3} \: \: \: \forall x \in [0, 3) 
\end{equation} 
It follows from \eqref{eq:expansion} that: 
\begin{eqnarray*}
e^{r\delta(t)} - 1 &=& r\delta(t) + \frac{(r\delta(t))^2}{2}g(r\delta(t)) \\
&\leq& r\delta(t) + \frac{(r\delta_{max})^2}{2}g(r\delta_{max}) \\
&\leq& r\delta(t) + \frac{(r\delta_{max})^2}{2(1-r\delta_{max}/3)} 
\end{eqnarray*}
where the final inequality uses \eqref{eq:g}, which  
is justified because $r\delta_{max}$ satisfies \eqref{eq:rcond1}. 
Thus: 
\begin{eqnarray*}
&&\hspace{-.4in}Y(t+1) - Y(t)\\
 &\leq& \left\{ \begin{array}{ll}
  Y(t)[r\delta(t) + \frac{(r\delta_{max})^2}{2(1-r\delta_{max}/3)}] &\mbox{ if $\norm{\bv{Q}(t)} \geq C$}  \\
e^{rC}[e^{r\delta_{max}} - 1]  & \mbox{ otherwise}  
\end{array}
\right.
\end{eqnarray*}
Taking expectations and using \eqref{eq:bracket} gives: 
\begin{eqnarray*}
&&\hspace{-.4in}\expect{Y(t+1) - Y(t)|\bv{Q}(t)}\\
 &\leq& \left\{ \begin{array}{ll}
  Y(t)[\frac{-r\epsilon}{2}  + \frac{(r\delta_{max})^2}{2(1-r\delta_{max}/3)}] &\mbox{ if $\norm{\bv{Q}(t)} \geq C$}  \\
e^{rC}[e^{r\delta_{max}} - 1]  & \mbox{ otherwise}  
\end{array}
\right.
\end{eqnarray*}

Now choose $r$ so that: 
\[ \frac{r\epsilon}{2} = \frac{(r\delta_{max})^2}{2(1-r\delta_{max}/3)} \]
This holds for $r$ as defined in \eqref{eq:r}, 
and this choice of $r$ maintains the inequality \eqref{eq:rcond1}. Thus: 
\begin{eqnarray*}
&&\hspace{-.3in}\expect{Y(t+1)-Y(t)|\bv{Q}(t)} \\
&&\leq  \left\{ \begin{array}{ll}
0 & \mbox{ if $\norm{\bv{Q}(t)}\geq C$} \\
e^{rC}[e^{r\delta_{max}} - 1] &\mbox{ otherwise} 
\end{array}
\right.
\end{eqnarray*}
Therefore, for all slots $t$: 
\[ \expect{Y(t+1) - Y(t)} \leq e^{rC}[e^{r\delta_{max}}-1] \]
Summing the above over $\tau \in \{0, 1, \ldots, t-1\}$ for some integer $t>0$ gives: 
\[\expect{Y(t)} - \expect{Y(0)} \leq e^{rC}[e^{r\delta_{max}}-1]t\]
Since $Y(0)=1$ with probability 1, and $Y(t) = e^{r\norm{\bv{Q}(t)}}$, one has: 
\[ \expect{e^{r\norm{\bv{Q}(t)}}} - 1 \leq  e^{rC}[e^{r\delta_{max}}-1]t\]
By Jensen's inequality for the convex function $e^x$ one has: 
\[ e^{r\expect{\norm{\bv{Q}(t)}}} - 1 \leq  e^{rC}[e^{r\delta_{max}}-1]t\]
Thus: 
\begin{eqnarray*}
 r\expect{\norm{\bv{Q}(t)}} &\leq& \log(1 +e^{rC}[e^{r\delta_{max}}-1]t)  \\
 &\leq& \max[\log(2), \log(2e^{rC}[e^{r\delta_{max}}-1]t)] \\
 &\leq& \max[\log(2), rC + \log(2t[e^{r\delta_{max}}-1])]
 \end{eqnarray*}
Dividing the above by $r$ gives the following, which holds for all integers $t>0$: 
\begin{eqnarray*}
\expect{\norm{\bv{Q}(t)}} \leq \max\left[\frac{\log(2)}{r}, C + \frac{\log(2t[e^{r\delta_{max}}-1])}{r}\right] 
\end{eqnarray*}
\end{proof} 

\section*{Appendix E --- The constant in Theorem \ref{thm:performance}} 

This appendix proves the inequality involving the $2BD$ constant 
at the end of the proof of Theorem \ref{thm:performance}.  From \eqref{eq:q-update} one has for all queues
$k \in \{1, 2, \ldots, K\}$ and all slots $\tau$: 
\[ |Q_k(\tau+1) - Q_k(\tau)| \leq |p_k(\tau-D) - c_k| \]
Thus, for all slots $t$: 
\begin{eqnarray*}
|Q_k(t+D) - Q_k(t)|  &\leq& \sum_{d=1}^D |Q_k(t+d) - Q_k(t+d-1)| \\
&\leq& \sum_{d=1}^D|p_k(t+d-1-D) - c_k| \\
&=& \sum_{d=1}^D|p_k(t_d) - c_k| \\
\end{eqnarray*}
where for notational simplicity $t_d$ has been defined: 
\[ t_d \defequiv t + d-1-D \]
Thus: 
\begin{eqnarray*}
&&\hspace{-.5in}\sum_{k=1}^K(Q_k(t+D)-Q_k(t))(p_k(t)-c_k) \\
&\leq& \sum_{k=1}^K\sum_{d=1}^D|p_k(t_d)-c_k||p_k(t)-c_k|
\end{eqnarray*}
Taking expectations of the above and using the Cauchy-Schwartz inequality:\footnote{Strictly speaking, these expectations should be conditioned on $\bv{Q}(t)$ to match with the inequalities at the end of Theorem \ref{thm:performance}.  That explicit conditioning has been suppressed to simplify the expressions.} 
\begin{eqnarray*}
&& \hspace{-.5in}\expect{\sum_{k=1}^K(Q_k(t+D)-Q_k(t))(p_k(t)-c_k)}  \\
&\leq& \sum_{k=1}^K\sum_{d=1}^D \sqrt{\expect{|p_k(t_d) - c_k|^2}}\sqrt{\expect{|p_k(t) - c_k|^2}} \\
&\leq& \sum_{d=1}^D \sqrt{\sum_{k=1}^K \expect{|p_k(t_d) - c_k|^2}} \sqrt{\sum_{k=1}^K \expect{|p_k(t) - c_k|^2}}
\end{eqnarray*}
where the final inequality follows because the inner product of two vectors 
is less than or equal to the product of norms.  The right hand side is less than or equal to: 
\[ \sum_{d=1}^D \sqrt{2B}\sqrt{2B}  = 2BD\]

\bibliographystyle{unsrt}
\bibliography{../../latex-mit/bibliography/refs}
\end{document}